\documentclass[12pt]{amsart}
\usepackage[utf8]{inputenc}

\usepackage[english]{babel}
\usepackage{amsmath, amssymb, amsthm, amscd,color,comment}
\usepackage{cancel}
\usepackage{cite}
\usepackage{alltt}
\usepackage[dvipsnames]{xcolor}
\usepackage{array}
\usepackage{caption} 
\usepackage{mathtools}

\usepackage{enumerate}
\usepackage{cancel}
\newtheorem{theorem}{Theorem}[section]

\newtheorem{example}[theorem]{Example}

\newtheorem{remark}[theorem]{Remark}

\newtheorem{proposition}[theorem]{Proposition}

\DeclarePairedDelimiter\floor{\lfloor}{\rfloor}

\def\cX{\mathcal X}

\def\cA{\mathcal A}

\def\a{\alpha}

\def\fqs{\mathbb F_{q^2}}

\def\fq{\mathbb F_q}

\def\deg{{\rm deg}}

\def\a{\alpha}

\newcommand{\al}{\alpha}

\begin{document}

\title{Reciprocal polynomials and curves with many points over a finite field}



\author{Rohit Gupta, Erik A. R. Mendoza and Luciane Quoos}

\address{Department of Mathematics, Birla institute of Technology and Science-Pilani, Hyderabad Campus, Hyderabad, Telangana, 500078, India}
\email{rohittgupta20@gmail.com}

\address{Instituto de Matemática, Universidade Federal do Rio de Janeiro, Cidade Universitária,
	CEP 21941-909, Rio de Janeiro, Brazil}
\email{luciane@im.ufrj.br}

\address{Instituto de Matemática, Universidade Federal do Rio de Janeiro, Cidade Universitária,
	CEP 21941-909, Rio de Janeiro, Brazil}
\email{erik@pg.im.ufrj.br}

\thanks{The research of Luciane Quoos and Erik A. R. Mendoza was partially supported by CNPq (Bolsa de Produtividade 302727/2019-1 and CNPq 141907/2020-7, respectively) and FAPERJ (SEI-260003/002364/2021 and FAPERJ 201.650/2021, respectively).}

\begin{abstract}
Let $\fqs$ be the finite field with $q^2$ elements. We provide a simple and effective method, using reciprocal polynomials, for the construction of algebraic curves over $\fqs$ with many rational points. The curves constructed are Kummer covers or fibre products of Kummer covers of the projective line. Further, we compute the exact number of rational points for some of the curves. 
\end{abstract}

\maketitle

\section{Introduction}

Let $\fq$ be the finite field with $q$ elements, where $q$ is a prime power. Let $\cX$ be a nonsingular, projective and absolutely irreducible algebraic curve defined over $\fq$. The celebrated theorem of Hasse-Weil states that
			$$|\# \cX (\fq)-q-1| \leq 2g(\cX) \sqrt{q}$$ 
where $\# \cX(\fq)$ is the number of $\fq$-rational points on the curve $\cX$ and  $g(\cX)$ is the genus of the curve $\cX$ over $\fq$.
A curve $\cX$ over $\fqs$ is called {\it maximal} if the number of $\fqs$-rational points $\# \cX(\fqs)$ attains the Hasse-Weil bound. It is well known that the genus of a maximal curve over $\fqs$ is at most $q(q-1)/2$.

In 1981, Goppa \cite{Goppa1981} introduced a way to associate an error-correcting code to a linear system on a curve over a finite field. In order to construct such good codes one requires curves with a large number of points. This leads to interest in study of curves over finite fields with many rational points with respect to their genus. Also such curves have applications in low-discrepancy sequences, stream ciphers, hash functions and finite geometries.

In the last few decades, several methods, such as class field theory, Drinfeld module, geometric and character theory, to find algebraic curves with many points have been studied \cite{RL2013, H2017, R2013, G2009, K2003, GQ2001, GG2003,  GV2000, XY2007, K2017, GV2015}. More explicit details about these methods can be found in \cite{GV1997}. However, the computation of the exact number of rational points on a given curve has always been a challenging problem and a general method to do such computations seems out of reach. Nevertheless, for certain very specific curves, some methods, such as evaluation of exponential sums and Kloosterman sums, as well as function field theory, have been helpful. For instance, Coulter \cite{C2002} used exponential sums to compute the number of rational points on a class of Artin-Schreier curves and Moisio \cite{M2007} used exponential sums and Kloosterman sums to compute the number of rational points on some families of Fermat curves. In \cite{OT2014,OTY2016}, the authors considered fibre products of Kummer covers of the projective line over $\fq$. In \cite{BM2020}, the authors gave a full description of the number of rational points in some extension $\mathbb{F}_{q^r}$ of $\fq$ in terms of Legendre symbol and quadratic characters for the Artin-Schreier curve $y^q-y= xP(x)-\lambda$ where $P(x) = x^{q^i}-x$ and $\lambda \in \fq$. For more details about these methods we refer to \cite{OT2014, OTY2016, M2007, C2002, BM2020}.

For a curve $\cX$ over $\fq$ with $g(\cX)\leq 50$, the webpage www.manypoints.org \cite{MP2009} collects the current intervals in which the number $\#\cX(\fq)$ is known to lie for some values of $q$. For a pair $(q, g)$, the tables there record an interval $[a, b]$ where $b$ is the best upper bound for the maximum number of points of a curve of genus $g$ over $\fq$, and $a$ gives a lower bound obtained from an explicit example of a curve $\cX$ defined over $\fq$ with $a$ (or at least $a$) rational points. At some places in manYPoints table in \cite{MP2009}, the lower bound $a$ of the interval $[a, b]$ is replaced by the symbol $`-$' where $`-$' represents the lower bound $L(q,g)$ given in Remark \ref{remark_manYPoints}.

In this article we improve upon the lower bounds of many of the intervals in \cite{MP2009} by constructing new examples of curves with many points. We provide a simple and effective construction of Kummer covers and fibre products of Kummer covers of the projective line over finite fields with many rational points using reciprocal polynomials. We give a general lower bound for the number of the rationals points under certain hypothesis, see Theorems \ref{theoX1}, \ref{theoX2}, \ref{theoX1X1} and \ref{theoX2X2}. In fact we calculate the exact number of rational points  for some particular constructions, see Theorem \ref{theo_N1}, Propositions \ref{propX1} and \ref{propX1_2}. 

As a consequence of these constructions, we obtain several improvements on the manYPoints table \cite{MP2009}. More precisely, we obtain the following $10$ new records.
	
\begin{enumerate}[(i)]
	\item A curve of genus $13$ over $\mathbb{F}_{7^4}$ with $3576$ rational points (see Example 4.7).
	\item A curve of genus $17$ over $\mathbb{F}_{7^4}$ with $3968$ rational points (see Example 4.7).
	\item A curve of genus $15$ over $\mathbb{F}_{17^2}$ with $708$ rational points (see Example 5.2).
	\item A curve of genus $10$ over $\mathbb{F}_{11^4}$ with $16952$ rational points (see Example 5.2).
	\item A curve of genus $15$ over $\mathbb{F}_{19^2}$ with $866$ rational points (see Example 5.2).
	\item A curve of genus $13$ over $\mathbb{F}_{17^2}$ with $648$ rational points (see Example 5.3).
	\item A curve of genus $8$ over $\mathbb{F}_{11^4}$ with $16566$ rational points (see Example 5.4).
	\item A curve of genus $22$ over $\mathbb{F}_{5^2}$ with $174$ rational points (see Example 6.2).
	\item A curve of genus $11$ over $\mathbb{F}_{13^2}$ with $444$ rational points (see Example 6.4).
	\item A curve of genus $13$ over $\mathbb{F}_{13^2}$ with $444$ rational points (see Example 6.4).
\end{enumerate}
We also obtain $119$ new entries and among these we point out three important new entries: explicit equations for maximal curves of genus $18$ and $36$ over $\mathbb{F}_{5^4}$ and analogously, an explicit equation for a maximal curve of genus $46$ over $\mathbb{F}_{7^4}$, see Remark \ref{novamaximal}.

	The remainder of this paper is organized as follows. In Section 2 we include the preliminaries and the notation used in the development of this work. In Section 3 we present a family of Kummer covers of the projective line over $\fqs$ defined by an affine equation of the type
			$$y^m=x^{\epsilon s} f(x)f^*(x)^{\lambda}$$ 
  $\epsilon, \lambda \in \{1, -1\}$, $s$ is a non-negative integer, $q=p^n$ with $p \nmid m$, $f(x)$ is a polynomial in $\fq[x]$ and $f^*(x)$ is the reciprocal polynomial of $f(x)$. We compute the genus of this family of curves, see Proposition \ref{prop1}. In Section 4, we study the particular case $\epsilon=-1$ and $\lambda=1$, and provide the exact number of rational points for some families of curves, see Theorem \ref{theo_N1}, Propositions \ref{propX1}  and \ref{propX1_2}. For certain parameters this families of curves turn out to be maximal over $\fqs$. We also give new examples of curves with many rational points, see Examples \ref{Example1}, \ref{Example2}, \ref{exmax2}, \ref{Example8}, \ref{Example9} and \ref{Example10}. In Section 5, we study the case $\epsilon=1$ and $\lambda=-1$. We again obtain new examples of curves with many points, see Examples \ref{Example3}, \ref{Example4}, \ref{Example5}, \ref{Example6} and \ref{Example7}. In Section 6, we obtain new curves with many rational points by considering the fibre products of the curves constructed in Sections 4 and 5, see Examples \ref{Example12} and \ref{Example11}. All the examples were obtained using the software Magma \cite{Magma}.

\section{Preliminaries and notation}

Throughout this article, we let $p$ be a prime number, $q$ a power of the prime $p$ and $\fq$ the finite field with $q$ elements.  For a nonsingular, projective, absolutely irreducible algebraic curve $\cX$ of genus $g(\cX)$ over $\fq$, $\fq (\cX)$ denotes its function field (where $\fq$ is its full constant field) and $\cX (\fq)$ denotes the set of $\fq$-rational points of the curve. For a function $z\in \fq (\cX)$, $(z)_{\fq(\cX)}$ stands for the principal divisor
of the function $z$ in $\fq(\cX)$.

We denote by $K=\overline{\fq}$ the algebraic closure of $\fq$. Moreover, we denote by $\xi_q$ a primitive element of the finite field $\fq$ and by $(a, b)$ the greatest common divisor of the elements $a$ and $b$ in a unique factorization ring. 

Given a polynomial  $f$  in $\fq[x]$ and a subset $\cA \subseteq \fqs$, we let $N_{f}(\cA):=\# \{\a\in \cA \mid f(\a)=0\}$ stand for  the number of roots of $f$ in  $\cA$.

We set some notation about curves with many points in the following remark.

\begin{remark}\label{remark_manYPoints}
		We say that a curve $\cX$ over $\fq$ with genus $g$ has {\it many points} if the number of $\fq$-rational points of $\cX$, denoted by $\# \cX(\fq)$, satisfies
		\begin{equation}\label{condition_manYPoints1} 
			\# \cX(\fq) \geq L(q,g):=\floor*{\frac{U(q, g)-q-1}{\sqrt{2}}}+q+1.
		\end{equation}  
		where $U(q, g)$ denotes the upper bound given in manYPoints table \cite{MP2009} for the number of $\fq$-rational points of a curve over $\fq$ with genus $g$.
		In particular, for a pair $(q,g)$ and a curve $\cX$ over $\fq$ with genus $g$, we say that $\cX$ gives a {\bf new record} (resp. {\bf meets the record}) if the number $\# \cX(\fq)$ is strictly larger than (resp. is equal to) the lower bound registered in manYPoints table corresponding to $(q,g)$. Further, we say that a curve $\cX$ over $\fq$ with genus $g$ is a {\bf new entry} if there was no earlier lower bound entry in manYPoints table corresponding to $(q,g)$ and $\# \cX(\fq)$ satisfies the relation (\ref{condition_manYPoints1}). 
		
		In the tables where we provide a {\it new record}, the notation OLB (old lower bound) stands for the lower bound on the number $\# \cX(\fq)$ of rational points  for a curve over $\fq$ of genus $g$ registered in the table \cite{MP2009}. And, in the tables where we provide a {\it new entry} the notation OLB stands for the lower bound given in (\ref{condition_manYPoints1}). Further, in the tables, the symbol  $\dagger$ indicates a maximal curve over $\fqs$.
\end{remark}

We finish this section by presenting the following remark that will be useful in the sequel. 
For a more general version see \cite[Theorems 3, 4]{OT2014}. We point out that throughout the article a rational point on the curve is the  same as a rational place in the function field of the curve.

\begin{remark}\label{remark1} 
	Let $\fq(x, y)/\fq$ be a Kummer extension of degree $m$ with full constant field $\fq$ and defined by the equation $y^m=h(x)$ where $m$ is a divisor of $q-1$ and $h \in \fq(x)$. For each  $\alpha\in \fq$, we write
	$$h(x)=(x-\alpha)^{k_{\alpha}} g_{\alpha}(x)$$
	where $k_{\alpha}\in \mathbb{Z}$, $g_{\alpha}\in \fq (x)$ and $g_{\alpha}(\alpha) \notin \{ 0, \infty\}$. Then there exist either no or exactly $(m, k_{\a})$ rational places of $\fq (x, y)$ over $P_{\a}$. In fact, there exists a rational place of $\fq (x, y)$ over $P_{\alpha}$ if and only if $g_{\alpha}(\alpha)$ is a $(m, k_{\alpha})$-power in $\fq^*$. 
	Moreover, suppose $$h(x)= c_{\infty} \frac{g_1(x)}{g_2(x)}$$
	where $c_{\infty}\in \fq^*$ and $g_1, g_2$ are monic polynomials in $\fq[x]$ with $(g_1, g_2)=1$. Then there exist either no or exactly $(m, \deg~g_2 - \deg~g_1)$ rational places of $\fq (x, y)$ over $P_{\infty}$. In the case of $P_{\infty}$, there exists a rational place of $\fq (x, y)$ over $P_{\infty}$ if and only if $c_{\infty}$ is a $(m, \deg~g_2 - \deg~g_1)$-power in $\fq^*$.
\end{remark}

\section{A construction of curves over $\fqs$.}\label{construction}
In this section we propose a construction of algebraic curves over $\fqs$ using reciprocal polynomials. We will see that certain specific polynomials provide interesting algebraic curves with many points. This idea is explored in more detail in the subsequent sections.

Given a polynomial $f(x)=a_0+a_1x+\cdots +a_dx^d \in \fq[x] $ of degree $d$, denote by $f^*(x)=x^df(1/x)$ the reciprocal polynomial of $f$.
For $m\geq 2$ an integer not divisible by $p$ and $s$ a non-negative integer, consider the algebraic curve $\cX$ over $\fqs$ defined by the affine equation
\begin{equation}\label{curveX}
	\cX: \quad y^m=x^{\epsilon s} f(x)f^*(x)^\lambda \text{ where } \epsilon,\lambda \in \{1, -1\}.
\end{equation}

With some assumptions on $f$, we compute the genus of these curves in the following proposition.
\begin{proposition}\label{prop1}
	Let $d>0$ and let $f(x)=a_0+a_1x+\cdots +a_dx^d \in \fq[x]$ be a separable polynomial of degree $d$ satisfying $f(0)\neq 0$. Let $s$ be a non-negative integer, $d_1$ be the degree of $(f, f^*)$ and $m \geq 2$ be such that $p \nmid m$. If $d_1<d$, then the algebraic function field $K(x, y)$ defined by the affine equation 
	$$y^m=x^{\epsilon s} f(x)f^*(x)^{\lambda}, \quad \text{where }  \epsilon, \lambda \in \{1, -1\},$$
	has genus
	$$
	g=(m-1)d+1-\frac{(m, s)+ (m, \epsilon s+d+\ d\lambda)+d_1(m,\lambda + 1)+d_1(m-2)}{2}.
	$$
\end{proposition}
\begin{proof}
	At first we write
	$$x^{\epsilon s} f(x)f^*(x)^{\lambda}=x^{\epsilon s}(f(x)/h(x))(f^*(x)/h(x))^{\lambda}h(x)^{1+\lambda}$$
	where $h=(f, f^*)$. The polynomials $h$, $f/h$ and $f^*/h$ are separable and $\alpha \in K$ is a root of $f$ if and only if $\alpha^{-1}$ is a root of $f^*$. So, without loss of generality, we can suppose  that
	$$
	f/h=\beta_1\prod_{i=1}^{d-d_1}(x-\alpha_i), \hspace{3mm}f^*/h=\beta_2\prod_{j=1}^{d-d_1}(x-\gamma_j)\hspace{3mm}\mbox{and}\hspace{3mm}h=\beta\prod_{k=d-d_1+1}^{d}(x-\alpha_k),
	$$
	where $\beta_1, \beta_2, \beta$ are in $K$, $\alpha_1, \alpha_2, \dots, \alpha_{d}$ are the roots of the polynomial $f$, $\gamma_j=\alpha_{i}^{-1}$ for some $1 \leq i\leq d$, and $\alpha_i, \gamma_j, \alpha_k$ are pairwise distinct for all $1\leq i, j\leq d-d_1$ and $d-d_1+1 \leq k \leq d$. The principal divisor of the function $x^{\epsilon s} f(x)f^*(x)^{\lambda}$ in $K(x)$ is given by
	\begin{align*}
		(x^{\epsilon s} f(x)f^*(x)^{\lambda})_{K(x)}&=\epsilon s (P_0-P_{\infty})+ \sum_{i=1}^{d-d_1}P_{\alpha_i}- (d-d_1)P_{\infty}+ \lambda \sum_{j=1}^{d-d_1}P_{\gamma_j}\\
		&\quad -\lambda(d-d_1)P_{\infty}+(\lambda+1)\sum_{k=d-d_1+1}^{d}P_{\alpha_k}-d_1(\lambda+1)P_{\infty}\\
		&=\epsilon s P_0+ \sum_{i=1}^{d-d_1}P_{\alpha_i}+\lambda \sum_{j=1}^{d-d_1}P_{\gamma_j}+(\lambda+1)\sum_{k=d-d_1+1}^{d}P_{\alpha_k}\\
		&\quad  -(\epsilon s +d + \lambda d)P_{\infty}.
	\end{align*}
	This implies that the extension $K(x, y)/K(x)$ is a Kummer extension of degree $m$ and for a place $P$ of $K(x, y)$, the ramification index is given by
	$$ 
	e(P)= \left\{
	\begin{array}{ll}
		m/ (m, s), \hspace{5mm} & \mbox{if }P \text{ is over } P_0,\\
		m, \hspace{5mm} & \mbox{if }P \text{ is over } P_{\alpha_i} \text{ or }P_{\gamma_i}, \text{ for } i=1,\dots, d-d_1,\\
		m/(m,\lambda+1 ), \hspace{5mm} & \mbox{if }P \text{ is over } P_{\alpha_i}, \text{ for } i=d-d_1+1,\dots, d,\\
		m/ (m, \epsilon s+d+\ d\lambda), \hspace{5mm} & \mbox{if }P \text{ is over } P_{\infty},\\
		1,  \hspace{25mm} & \mbox{otherwise}.
	\end{array} \right .
	$$ 
	By  Riemann-Hurwitz formula, the genus $g$ of $K(x, y)$ satisfies
	$$2g-2=-2m+m-(m,s)+2(m-1)(d-d_1)+d_1(m-(m,\lambda+1))+m- (m, \epsilon s+d+\ d\lambda),$$
	which gives
	$$
	g=(m-1)d+1-\frac{(m, s)+ (m, \epsilon s+d+\ d\lambda)+d_1(m,\lambda + 1)+d_1(m-2)}{2}.
	$$
\end{proof}
In the subsequent sections, we investigate the number of $\fqs$-rational points on the curve (\ref{curveX}) for the cases  $\epsilon=-1$ and $\lambda=1$, and $\epsilon=1$ and $\lambda=-1$ separately. Note that the curves $\cX$ for $\epsilon=\lambda=1$ and  $\epsilon=\lambda=-1$ are isomorphic to the curves with $\epsilon=-1$ and $\lambda=1$, and $\epsilon=1$ and $\lambda=-1$ respectively.

\section{Curves over $\fqs$ from Section \ref{construction}: the case of $\epsilon=-1$ and $\lambda=1$.}\label{Section 4}

In this section, we restrict ourselves to the curve $\cX$ in (\ref{curveX}) with $\epsilon=-1$ and $\lambda=1$. We impose certain conditions on the polynomial $f \in \fq[x]$ to provide an explicit expression and a lower bound for the number of $\fqs$-rational points on the curve $\cX$. Moreover, for some of these algebraic curves, we compute the exact number of $\fqs$-rational points, see Theorem \ref{theo_N1}, Propositions \ref{propX1} and \ref{propX1_2}.

\begin{theorem}\label{theoX1}
Let $m \geq 2$ be a divisor of $q+1$, $f \in \fq[x]$ be a separable polynomial of degree $d$ satisfying $f(0)\neq 0$ and $(f, f^*)=1$, and $s$ be an integer $0\leq s<m$. Then the algebraic curve defined by 
\begin{equation}\label{curveX1}
\cX: \quad y^{m}=\frac{f(x)f^*(x)}{x^s}
\end{equation}
has genus 
$$g=(2md-2(d-1)-(m, s)-(m, 2d-s))/2.$$ 
Further if $(f, x^{q+1}-1)=1$, then the number of rational points $\#\cX(\fqs)$ over $\fqs$ satisfies
\begin{align*}
\#\cX(\fqs) &\geq m((q+1,2(d-s))+q-3-2N_{f}(\fq^*))+2N_f(\fqs).
\end{align*}
In particular, for $s=d$, we have 
$
\#\cX(\fqs) \geq 2m(q-1-N_f(\fq^*))+2N_f(\fqs).
$
\end{theorem}
\begin{proof}
A direct application of Proposition \ref{prop1} gives the genus of the curve defined in (\ref{curveX1}). We now provide an expression for the number of $\fqs $-rational points on this curve . Let $\alpha \in \fqs^*$ be such that $f(\a)f^*(\a)\neq 0$, then $\frac{f(\a)f^*(\a)}{\a^s}$ is a $m$-power in $\fqs$ if and only if
$
\left(\frac{f(\a)f^*(\a)}{\a^s}\right)^{\frac{q^2-1}{m}}=1,$
which is equivalent to
$$ \left(\left( \frac{f(\a)f^*(\a)}{\a^s}\right)^{q-1}-1\right)\left(\sum_{i=0}^{\frac{q+1}{m}-1}\left(\frac{f(\a)f^*(\a)}{\a^s}\right)^{(q-1)i}\right)=0, $$
that is,
$$((f(\a)f^*(\a))^{q-1}-\a^{s(q-1)})\left(\sum_{i=0}^{\frac{q+1}{m}-1}(f(\a)f^*(\a))^{(q-1)i}\a^{s(q-1)\left(\frac{q+1}{m}-1-i\right)}\right)=0.$$
Let
$$
h_1(x)=(f(x)f^*(x))^{q-1}-x^{s(q-1)}\quad \text{ and} \quad
h_2(x)=\sum_{i=0}^{\frac{q+1}{m}-1}(f(x)f^*(x))^{(q-1)i}x^{s(q-1)\left(\frac{q+1}{m}-1-i\right)}.$$
Then $h_1$ and $h_2$ are co prime polynomials. In fact, if $\a$ is a root of $h_1$, then $(f(\a)f^*(\a))^{q-1}=\a^{s(q-1)}$ and
\begin{align*}
h_2(\a)&=\sum_{i=0}^{\frac{q+1}{m}-1}(f(\a)f^*(\a))^{(q-1)i}\a^{s(q-1)\left(\frac{q+1}{m}-1-i\right)}\\
&=\sum_{i=0}^{\frac{q+1}{m}-1}\a^{s(q-1)i}\a^{s(q-1)\left(\frac{q+1}{m}-1-i\right)}\\
&=\left(\frac{q+1}{m}\right)\a^{s(q+1)\left(\frac{q+1}{m}-1\right) } \neq 0.
\end{align*}
It is also clear that $(h_1, ff^*)=(h_2, ff^*)=1$. 
We conclude that
 $$\#\{ \a \in \fqs^* \mid f(\a)f^*(\a)\neq 0 \text{ and }\frac{f(\a)f^*(\a)}{\a^s} \text{ is a } m\text{-th } \text{power in } \fqs^*\}=N_{h_1}(\fqs)+N_{h_2}(\fqs).$$
From Remark \ref{remark1}, each $\a \in \fqs$ such that $f(\a)f^*(\a)=0$ gives one rational point on the curve. From Remark \ref{remark1}, we also conclude that each one of $x=0$ and $x=\infty$ contributes $(m , s)$ and $(m, 2d-s)$ rational points on the curve respectively.
So the number of rational points on the curve $\cX$ is
\begin{equation}\label{NRPX1}
\# \cX (\fqs)=(m, s)+(m, 2d-s)+2N_f(\fqs)+m(N_{h_1}(\fqs)+N_{h_2}(\fqs)).
\end{equation}
 Now we assume that $(f, x^{q+1}-1)=1$. Note that for $\beta \in \{ x \in \fqs \mid x^{(q+1, 2(d-s))}=1\}$, we have $\beta^q=\beta^{-1}$ and thus we write
\begin{align*}
h_1(\beta)&=(f(\beta)f^*(\beta))^{q-1}-\beta^{s(q-1)}=\frac{(f(\beta)f^*(\beta))^q}{f(\beta)f^*(\beta)}-\beta^{-2s}\\&=\frac{f(\beta^q)f^*(\beta^q)}{f(\beta)f^*(\beta)}-\beta^{-2s}=\frac{f(1/\beta)f^*(1/\beta)}{f(\beta)f^*(\beta)}-\beta^{-2s}\\
&=\frac{f(\beta)f^*(\beta)}{\beta^{2d}f(\beta)f^*(\beta)}-\beta^{-2s}=\beta^{-2d}-\beta^{-2s}=0.
\end{align*}
Also, for $\beta \in \fq^*$ such that $f(\beta)f^*(\beta)\neq 0$, we have $h_1(\beta)=0$. Therefore
\begin{align*}
N_{h_1}(\fqs)&\geq (q+1,2(d-s))+q-1-2N_{f}(\fq^*)-(q-1, 2).
\end{align*}
Hence we get
$$
\# \cX (\fqs) \geq  2N_f(\fqs)+m((q+1,2(d-s))+q-3-2N_{f}(\fq^*)).
$$
\end{proof}

In what follows we compute the genus and the exact number of rational points for some families of algebraic curves as constructed in (\ref{curveX1}).
\begin{theorem}\label{theo_N1}
Let $b \in \fq^*, b^2\neq 1$ and $d$ be a positive divisor of $q+1$. Then the algebraic curve  defined by 
$$ 
\cX \,:\, y^{q+1}=\frac{bx^{2d}+(b^2+1)x^{d}+b}{x^{d}}
$$
has genus $$g=d(q-1)+1$$  and  its number of $\fqs$-rational points is given by
$$\# \cX(\fqs)=d(q^2-1)+(d, 2)(q+1)^2+4d-d(q+1)((q-1, 2)+2).
$$
In particular, if $q$ is odd, $d=2$ and $q\geq 17$, then this curve has many points.
\end{theorem}
\begin{proof}
The curve $\cX$ corresponds to the construction in  (\ref{curveX1}) for $f(x)=x^d+b$, $s=d$ and $m=q+1$. Since $b^2\neq 1$, we have $(f, f^*)=1$. The genus of the curve follows from Theorem \ref{theoX1}. Now we compute the number of $\fqs$-rational points on the curve following the proof and notation as in Theorem \ref{theoX1}.
Since $b\in\fq^*$ and $d$ is a divisor of $q+1$, each one of $f$ and $f^*$ has $d$ distinct roots in $\fqs^*$ and therefore $N_{ff^*}(\fqs)=2d$. Note that each root of $ff^*$ contributes one rational point on the curve. From Remark \ref{remark1},  we also conclude that each one of $x=0$ and $x=\infty$ contributes $(q+1, d)=d$ rational points on the curve respectively.

For $\a \in \fqs^*$ and $ f(\a)f^*(\a)\neq 0$, we have
$$\#\{ \a \in \fqs^* \mid f(\a)f^*(\a)\neq 0 \text{ and }\frac{f(\a)f^*(\a)}{\a^d} \text{ is a } (q+1)\text{-power in } \fqs^*\}=N_{h_1}(\fqs)+N_{h_2}(\fqs)$$
where
$$h_1(x)=(f(x)f^*(x))^{q-1}-x^{d(q-1)}=\frac{b(x^{d(q+1)}-1)(x^{d(q-1)}-1)}{f(x)f^*(x)}\quad \text{and}\quad h_2(x)\equiv 1.$$
Clearly $N_{h_2}(\fqs)=0$. Next we show that the polynomial $h_1\in \fq[x]$ has
$d(q-1)+2(q+1)-4d$ distinct roots in $\fqs^*$. In fact, since 
\begin{align*}
	&(x^{d(q+1)}-1, x^{q^2-1}-1)=x^{(d, 2)(q+1)}-1,\\ 
	&(x^{d(q-1)}-1, x^{q^2-1}-1)=x^{d(q-1)}-1\\
 \textrm{and} \quad	&(x^{(d, 2)(q+1)}-1, x^{d(q-1)}-1)=x^{d(q-1, 2)}-1,
\end{align*}
we obtain $d(q-1)+(d, 2)(q+1)-d(q-1, 2)$ distinct roots of $(x^{d(q+1)}-1)(x^{d(q-1)}-1)$ in $\fqs^*$. Since $N_{ff^*}(\fqs)=2d$, we conclude that $h_1$ has
$d(q-1)+(d, 2)(q+1)-d(q-1, 2)-2d$ distinct roots in $\fqs^*$. 
Hence the number of $\fqs$-rational points on the curve $\cX$ is given by
\begin{align*}
\#\cX (\fqs)&=4d+(q+1)(d(q-1)+(d, 2)(q+1)-d(q-1, 2)-2d)\\
&=d(q^2-1)+(d, 2)(q+1)^2+4d-d(q+1)((q-1, 2)+2).
\end{align*}

Next we show that for $q$ odd, $d=2$ and $q \geq 17$ this curve has many points. By Remark \ref{remark_manYPoints}, a curve is considered to have many points if and only if $ L(q^2, g) \leq \# \cX(\fqs)$. From the Hasse-Weil bound, we have
$$
L(q^2, g)\leq \floor*{\frac{q^2+1+2 g q-q^2-1}{\sqrt{2}}}+q^2+1=\floor*{\sqrt{2}gq}+q^2+1 .
$$
In particular, algebraic curves satisfying $\sqrt{2}gq+q^2+1 \leq  \# \cX(\fqs)$ have many points. 
Therefore the curve $\cX$ has many points if
\begin{equation}\label{cond2}
\sqrt{2}q(d(q-1)+1)+q^2+1 \leq d(q^2-1)+(d, 2)(q+1)^2+4d-d(q+1)((q-1, 2)+2). 
\end{equation}
The condition (\ref{cond2}) is never satisfied when $q$ is even or when $q$ is odd and $d \neq 2$. For $q$ odd and $d=2$, this condition is satisfied if and only if $q\geq 17$.
\end{proof}

Note that for $b^2=1$ the curve $\cX$ in Theorem 4.2 is isomorphic to the curve with affine equation $y^{q+1}=(x^2+b)^2/x^d$. In order to complete the analysis of the curve $\cX$ given in Theorem \ref{theo_N1}, we study an absolutely irreducible component of curve obtained when $d$ is even in Proposition \ref{propX1_2} and for $d$ odd, we study this curve in Proposition \ref{propX1}.

\begin{proposition}\label{propX1_2} 
	Assume $q$ is odd. Let $d\geq 1$ be a positive integer such that $4d$ divides $q^2-1$ and $b\in \fq$ be  such that $b^2=1$. Then the algebraic curve $\cX$ defined by the affine equation 
	\begin{equation}\label{eqTheo_N2}
		y^{(q+1)/2}=\frac{x^{2d}+b}{x^{d}}
	\end{equation}
	has genus 
	$$g=\frac{d(q-1)+2-(2d, q+1)}2$$
	and its number of $\fqs$-rational points is given by
	$$
	\# \cX (\fqs) = \frac{(q+1)^2(2d, q-1)+(q^2+1)(2d, q+1)-2d(3q+1)}{2}.
	$$
	In particular, this curve is maximal over $\fqs$ if and only if $(2d, q+1)+(2d, q-1)=2(d+1)$.
\end{proposition}
\begin{proof}
	By Remark \ref{remark1}, each one of the points $x=0$ and $x=\infty$ contributes with $(d, \frac{q+1}{2})$ rational points on the curve. Now we consider the roots of $x^{2d}+b$. Since $4d\mid q^2-1$, we have 
	$\# \{ \a \in \fqs \mid \a^{2d}+b=0\}=2d$ and each root of $x^{2d}+b$ contributes with one rational point.
	On the other hand, for $\a\in \fqs^*$ such that $\a^{2d}+b\neq 0$, we have 
	\begin{align}\label{cond1}
		\frac{\al^{2d}+b}{\alpha^{d}} \text{ is a } \frac{(q+1)}{2}\text{-power in }\fqs^* 
		&\Leftrightarrow \left(\frac{\a^{2d}+b}{\a^{d}}\right)^{2(q-1)} = 1\\ \nonumber
		&\Leftrightarrow (\a^{2d(q+1)}-1)(\alpha^{2d(q-1)}-1)=0. \nonumber
	\end{align}
	Since
	\begin{align*}
		&(x^{2d(q+1)}-1, x^{q^2-1}-1)=x^{(q+1)(2d, q-1)}-1,\\ &(x^{2d(q-1)}-1, x^{q^2-1}-1)=x^{(q-1)(2d, q+1)}-1\\
		\textrm{and} \quad &(x^{(q+1)(2d, q-1)}-1, x^{(q-1)(2d, q+1)}-1)=x^{4d}-1,
	\end{align*}
	we obtain that there are $(q+1)(2d, q-1)+(q-1)(2d, q+1)-4d$ elements $\alpha \in \fqs^*$ satisfying (\ref{cond1}).
	Also, since 
	$$
	x^{2d}+b \mid (x^{2d(q+1)}-1)(x^{2d(q-1)}-1),
	$$
	we conclude that the polynomial $(x^{2d(q+1)}-1)(x^{2d(q-1)}-1)$ has $(q+1)(2d, q-1)+(q-1)(2d, q+1)-6d$ distinct roots in $\fqs^*\setminus \{\al \in \fqs^* : \a^{2d}+b=0 \}$. Consequently, 
	\begin{align*}
		\# \cX (\fqs) &= 2d+2(d, \frac{q+1}{2})+\frac{q+1}{2}((q+1)(2d, q-1)+(q-1)(2d, q+1)-6d)\\
		&=\frac{(q+1)^2(2d, q-1)+(q^2+1)(2d, q+1)-2d(3q+1)}{2}.
	\end{align*}
	Finally, we note that
	$$
	\# \cX (\fqs)-q^2-1-2gq=\frac{(q+1)^2}{2}((2d, q-1)+(2d, q+1)-2-2d).
	$$
	This completes the proof.
\end{proof}

\begin{proposition}\label{propX1}
Assume $q$ is odd. Let $d\geq 1$ be an odd integer such that $p\nmid d$ and $b\in \fq$ such that $b^2=1$. Then the algebraic curve $\cX$ defined by the equation
\begin{equation}\label{curveX11}
y^{q+1}=\frac{(x^d+b)^2}{x^d}\\
\end{equation}
has genus 
$$g=\frac{d(q-1)+2-2(d, q+1)}2$$ 
and its number of $\fqs$-rational points is given by
$$
\# \cX(\fqs)=(q^2+1)(d, q+1)+(q+1)^2(d, q-1)-(3q+1)(d, q^2-1).
$$
In particular,  for a divisor $d$ of $q^2-1$, the curve $\cX$ is $\fqs$-maximal if and only if either $(d,q+1)=1$  or $(d,q-1)=1$.
\end{proposition}
\begin{proof}
The computation of the genus is analogous to the one in Proposition \ref{prop1} and the computation of the number of $\fqs$-rational points on the curve is analogous to the proof of Proposition \ref{propX1_2}.  For a divisor $d$ of $q^2-1$, we have
\begin{align*}
\# \cX (\fqs)-q^2-1-2gq&=(q^2+1)(d, q+1)+(q+1)^2(d, q-1)-(3q+1)d-q^2-1\\
&\quad -dq(q+1)+2dq+2q(d, q+1)-2q\\
&= (q+1)^2((d, q+1)+(d, q-1)-1)-d(q+1)^2\\
&= (q+1)^2((d, q+1)+(d, q-1)-1-d)\\
&= -(q+1)^2((d, q+1)-1)((d, q-1)-1).
\end{align*}
This completes the proof.
\end{proof}
\begin{remark}The curve in Proposition \ref{propX1} is isomorphic to the curve $$y^{q+1}=x^{q+1-d}(x^d+b)^2.$$ We point out that for some values of $d$ (for instance, when $d$ is a divisor of $q+1$), this curve has been appeared in \cite[Example 6.4 (case 2)]{GSX2000} as a subcover of the Hermitian curve over $\fqs$ given by 
$$y^{m_1}=(-1)^kx^{bm}(x^m+1)^k$$
where $m$, $m_1$ are divisors of $q+1$, and $k$, $b$ are positive integers. 
\end{remark}


We now provide examples of curves with many points from the constructions obtained in this section. 
\begin{example}\label{Example1}
Let $f(x)=x+b$ where $b \in \fq^*$ such that $ b^2\ne 1$. In the following tables, we list $q,m,b,s,g$ and $\# \cX(\fqs)$ where $m, s,\cX$ and $f$ satisfy the hypothesis of Theorem \ref{theoX1}.

		\begin{minipage}[c]{0.4\textwidth}
			\centering
			\captionof*{table}{\bf{Meet record}}
			\begin{tabular}{
|>{\centering\arraybackslash}p{0.5cm}|>{\centering\arraybackslash}p{0.5cm} |>{\centering\arraybackslash}p{0.5cm} |>{\centering\arraybackslash}p{0.5cm} |>{\centering\arraybackslash}p{0.5cm} |>{\centering\arraybackslash}p{1.6cm} |
}	
\hline 
			$q$ & $m$ & $b$ & $s$ & $g$ & $\# \cX(\fqs)$\\ \hline           
$3^2$   &  $5$ & $\xi_{3^2}^2$ & $3$  &  $4$  &  $154^\dagger$ \\ \hline
$3^2$   &  $10$ & $\xi_{3^2}^2$ & $4$  &  $8$  &  $226^\dagger$ \\ \hline
			$5$  &   $6$  &  $2$ &  $4$ &  $4$ & $66^\dagger$ \\ \hline
			$5^2$  &   $13$  & $\xi_{5^2}$ & $0$ & $6$ &  $926^\dagger$ \\ \hline
			$5^2$  &   $13$  & $2$ & $1$ & $12$ &  $1226^\dagger$ \\ \hline
			
$5^2$  &   $26$  & $2$ & $4$ & $24$ &  $1826^\dagger$ \\ \hline
			$7$ & $4$ & $2$ & $3$ &  $3$ & $92^\dagger$ \\ \hline
						$7^2$ & $5$ & $\xi_{7^2}^3$ & $1$ &  $4$ & $2794^\dagger$ \\ \hline
			$7^2$ & $10$ & $3$ & $4$ &  $8$ & $3186^\dagger$ \\ \hline
			$7^2$ & $25$ & $\xi_{7^2}$ & $0$ &  $12$ & $3578^\dagger$ \\ \hline
			$7^2$ & $50$ & $\xi_{7^2}^{12}$ & $10$ &  $44$ &$6714^\dagger$ \\ \hline
			$7^2$ & $50$ & $\xi_{7^2}^{12}$ & $4$ &  $48$ & $7106^\dagger$ \\ \hline
$13$ & $14$ & $2$ & $0$ &  $6$ & $326^\dagger$ \\ \hline
			$13$ & $14$ & $5$ & $4$ &  $12$ & $482^\dagger$ \\ \hline
		\end{tabular}
		\end{minipage}
		\hfill
		\begin{minipage}[c]{0.6\textwidth}
		\centering
			\captionof*{table}{\bf{Meet record}}
			\begin{tabular}{
|>{\centering\arraybackslash}p{0.5cm}|>{\centering\arraybackslash}p{0.5cm} |>{\centering\arraybackslash}p{0.5cm} |>{\centering\arraybackslash}p{0.5cm} |>{\centering\arraybackslash}p{0.5cm} |>{\centering\arraybackslash}p{1.6cm} |
}	
\hline 
			$q$ & $m$ & $b$ & $s$ & $g$ & $\# \cX(\fqs)$\\ \hline           
			$13^2$ &  $10$ &  $\xi_{13^2}^4$ &  $3$  &  $9$& $31504$\\ \hline
			$13^2$ &  $34$ &  $8$  & $4$  &  $32$ &   $39378^\dagger$ \\ \hline
			$17$ &  $18$ &  $2$  & $0$  &  $8$ &   $562^\dagger$ \\ \hline
			$17$ &  $18$ &  $4$  & $6$  &  $14$ &   $766^\dagger$ \\ \hline
			$17$ &  $18$ &  $4$  & $4$  &  $16$ &   $834^\dagger$ \\ \hline
			$19$ &  $5$ &  $2$  & $0$  &  $2$ &   $438^\dagger$ \\ \hline
			$19$ &  $20$ &  $2$  & $0$  &  $9$ &   $704^\dagger$ \\ \hline
		\end{tabular}
		\vspace{1cm}
			\centering
			\captionof*{table}{\bf{New entry}}			
			\begin{tabular}
{
|>{\centering\arraybackslash}p{0.5cm}|>{\centering\arraybackslash}p{0.5cm} |>{\centering\arraybackslash}p{0.5cm} |>{\centering\arraybackslash}p{0.5cm} |>{\centering\arraybackslash}p{0.5cm} |>{\centering\arraybackslash}p{1.6cm} |
>{\centering\arraybackslash}p{1.6cm} |
}				
			\hline
				$q$ & $m$ & $b$ & $s$ & $g$ & $\# \cX(\fqs)$ & $\text{OLB}$\\ \hline 
				$ 7^2 $ & $ 50 $ & $ \xi_{7^2}^{4} $ & $5$ & $ 47 $ & $ 5708 $ & $ 5658 $\\ \hline
		  \end{tabular}
		\end{minipage}
\end{example}

\begin{example}\label{Example2}
	Let $f(x)=x^2+b$ where $b \in \fq^*$ such that $b^2\ne 1$. We list $q,m,b,s,g$ and $\# \cX(\fqs)$ in the following tables where $m, s,\cX$ and $f$ satisfy the hypothesis of Theorem \ref{theoX1}. We note that if $m=q+1$ and $s=d$ in the following tables, then the genus $g$ and the number of $\fqs$-rational points $\# \cX(\fqs)$ satisfies Theorem \ref{theo_N1}.
	
		\begin{minipage}[c]{0.4\textwidth}
		\centering
		\captionof*{table}{\bf{Meet record}}
		\begin{tabular}
{
|>{\centering\arraybackslash}p{0.5cm}|>{\centering\arraybackslash}p{0.5cm} |>{\centering\arraybackslash}p{0.5cm} |>{\centering\arraybackslash}p{0.5cm} |>{\centering\arraybackslash}p{0.5cm} |>{\centering\arraybackslash}p{1.6cm} |
}		
		\hline
			$q$ & $m$ & $b$ & $s$ & $g$ & $\# \cX(\fqs)$\\ \hline     
			$3^2$ &  $5$ & $\xi_{3^2}$ &  $0$ &   $6$ &  $190^\dagger$\\ \hline           
			$3^2$ & $10$ & $\xi_{3^2}$ & $2$ & $17$ & $288$ \\ \hline   
			$5$ &  $6$ &  $2$ &  $5$ &   $10$ &   $126^\dagger$ \\ \hline
			$5^2$ &  $26$ &  $\xi_{5^2}$ &  $2$ &   $49$ &   $2400$ \\ \hline      
			$7^2$ &  $5$ &  $\xi_{7^2}^3$ &  $0$ &   $6$ &   $2990^\dagger$ \\ \hline           
			$7^2$ &  $25$ &  $3$ &  $0$ &   $36$ &   $5930^\dagger$ \\ \hline               
			$11$ & $2$ & $3$ & $1$ & $2$ & $166^\dagger$ \\ \hline  
			$11$ & $3$ & $3$ & $2$ & $4$ & $210^\dagger$ \\ \hline  
			$11$ &  $6$ &  $3$ &  $5$ &   $10$ &   $342^\dagger$ \\ \hline     
	    	$13$ &  $2$ &  $5$ &  $1$ &   $2$ &   $222^\dagger$ \\ \hline     
			$13$ &  $14$ &  $2$ &  $2$ &   $25$ &   $624$ \\ \hline       
			$13$ &  $14$ &  $5$ &  $9$ &   $26$ &   $846^\dagger$ \\ \hline
			$13^2$ & $5$ & $8$ & $2$ & $8$ & $31266^\dagger$ \\ \hline
			$13^2$ & $10$ & $5$ & $2$ & $17$ & $34208$ \\ \hline
			$17$ & $2$ & $3$ & $1$& $2$ & $358^\dagger$ \\ \hline     
			$17$ & $3$ & $3$ & $2$& $4$ & $426^\dagger$ \\ \hline
		    $17$ &  $6$ &  $3$ &  $5$ &   $10$ &   $630^\dagger$ \\ \hline 
			$17$ &  $9$ &  $5$ &  $0$ &   $12$ &   $698^\dagger$ \\ \hline
		\end{tabular}
	\end{minipage}
\hfill
\begin{minipage}[c]{0.6\textwidth}	
\centering
		\captionof*{table}{\bf{Meet record}}
		\begin{tabular}
{
|>{\centering\arraybackslash}p{0.5cm}|>{\centering\arraybackslash}p{0.5cm} |>{\centering\arraybackslash}p{0.5cm} |>{\centering\arraybackslash}p{0.5cm} |>{\centering\arraybackslash}p{0.5cm} |>{\centering\arraybackslash}p{1.6cm} |
}		
		\hline
			$q$ & $m$ & $b$ & $s$ & $g$ & $\# \cX(\fqs)$\\ \hline     
			$17^2$ & $5$ & $4$ & $2$ & $8$ & $88146^\dagger$\\ \hline        
			$19$ &  $5$ &  $4$ &  $0$ &   $6$ &   $590^\dagger$ \\ \hline           
			$19$ &  $10$ &  $14$ &  $5$ &   $16$ &   $970^\dagger$ \\ \hline     
		\end{tabular}
		\vspace{0.7cm}
		\centering
		\captionof*{table}{\bf{New entry}}			
		\begin{tabular}
{
|>{\centering\arraybackslash}p{0.5cm}|>{\centering\arraybackslash}p{0.5cm} |>{\centering\arraybackslash}p{0.5cm} |>{\centering\arraybackslash}p{0.5cm} |>{\centering\arraybackslash}p{0.5cm} |>{\centering\arraybackslash}p{1.6cm} |
>{\centering\arraybackslash}p{1.6cm} |
}		
		\hline
			$q$ & $m$ & $b$ & $s$ & $g$ & $\# \cX(\fqs)$ & $\text{OLB}$\\ \hline  
		$7^2$ & $25$ & $\xi_{7^2}^3$ & $5$ & $46$ & $6910^\dagger$ & $5589$\\ \hline     
		$13^2 $ & $ 34 $ & $\xi_{13^2}^5$ & $0$ & $49$ & $41112 $ & $40273$\\ \hline
			$17 $ & $ 18 $ & $2$ & $2$ & $ 33 $ & $ 1088 $ & $1083$\\ \hline 
			$17^2 $ & $ 10 $ & $5$ & $2$ & $17$ & $92928 $ & $90470$\\ \hline   
		$19 $ & $ 20 $ & $2$ & $2$ & $ 37 $ & $ 1368 $ & $1356$\\ \hline 
		\end{tabular}
		\vspace{0.7cm}
		\centering
		\captionof*{table}{\bf{New record}}
		\begin{tabular}
{
|>{\centering\arraybackslash}p{0.5cm}|>{\centering\arraybackslash}p{0.5cm} |>{\centering\arraybackslash}p{0.5cm} |>{\centering\arraybackslash}p{0.5cm} |>{\centering\arraybackslash}p{0.5cm} |>{\centering\arraybackslash}p{1.6cm} |
>{\centering\arraybackslash}p{1.6cm} |
}		
		\hline
			$q$ & $m$ & $b$ & $s$ & $g$ & $\# \cX(\fqs)$ & $\text{OLB}$\\ \hline 
			$7^2$ & $ 10 $ & $ \xi_{7^2}^3 $ & $0$ & $ 13 $ & $ 3576 $ & $3258 $\\ \hline
		$7^2$ & $ 10 $ & $ \xi_{7^2}^3 $ & $2$ & $ 17 $ & $ 3968 $ & $3808$\\ \hline
		\end{tabular}
	\end{minipage}
\end{example}

\begin{example}\label{exmax2}
Let $f(x)=x^3+b \in \fq[x]$, $m \geq 2$ be a divisor of $q+1$ and $s$ be an integer $0\leq s<m$. We consider the algebraic curve defined by 
\begin{equation*}
	\cX: \quad y^{m}=\frac{f(x)f^*(x)}{x^s}.
\end{equation*}
The following tables consists of $q,m,b,s,g$ and $\# \cX(\fqs)$ which leads to meet record/new entry in the manYPoints table in \cite{MP2009}. Further, if $m=q+1$, $s=d$ and $b^2=1$ in the following tables, then the genus $g$ and the number of $\fqs$-rational points $\# \cX(\fqs)$ satisfies Proposition \ref{propX1}.

	\begin{minipage}[c]{0.4\textwidth}
		\centering
		\captionof*{table}{\bf{Meet record}}
		\begin{tabular}
{
|>{\centering\arraybackslash}p{0.5cm}|>{\centering\arraybackslash}p{0.5cm} |>{\centering\arraybackslash}p{0.5cm} |>{\centering\arraybackslash}p{0.5cm} |>{\centering\arraybackslash}p{0.5cm} |>{\centering\arraybackslash}p{1.6cm} |
>{\centering\arraybackslash}p{1.6cm} |
}		
		\hline
			$q$ & $m$ & $b$ & $s$ & $g$ & $\# \cX(\fqs)$\\ \hline
			$17$ & $18$  & $4$ & $0$ & $40$ &    $1650^\dagger$ \\    \hline           	
		\end{tabular}
	\end{minipage}
	\hfill
	\begin{minipage}[c]{0.6\textwidth}
		\centering
		\captionof*{table}{\bf{New entry}}			
		\begin{tabular}
{
|>{\centering\arraybackslash}p{0.5cm}|>{\centering\arraybackslash}p{0.5cm} |>{\centering\arraybackslash}p{0.5cm} |>{\centering\arraybackslash}p{0.5cm} |>{\centering\arraybackslash}p{0.5cm} |>{\centering\arraybackslash}p{1.6cm} |
>{\centering\arraybackslash}p{1.6cm} |
}		
		\hline
			$q$ & $m$ & $b$ & $s$ & $g$ & $\# \cX(\fqs)$ & $\text{OLB}$\\
			\hline
			$5^2$ & $13$ &  $1$  &   $3$ & $18$ &   $1526^\dagger$   &   $1262$   \\ \hline
			$5^2$ & $26$ &  $1$  &   $3$ & $36$ &   $2426^\dagger$   &   $1898$   \\ \hline
			$7^2$ & $10$ & $\xi_{7^2}^2$ &   $8$ & $26$ &  $4444$    &  $4203$   \\ \hline
			$7^2$ & $10$ & $\xi_{7^2}^3$	&   $3$ & $27$ & $4748$     & $4273$    \\ \hline
			$13^2$& $10$ & $\xi_{{13}^2}^2$ & 	$8$ & $26$ & $36604$    & $34776$  \\
			\hline 
		\end{tabular}
		\centering
	\end{minipage}
\end{example}

\begin{remark}\label{novamaximal}
For $q=5^2$ in Example \ref{exmax2}, we obtain an explicit equation for a maximal curve of genus $36$ over $\mathbb{F}_{5^4}$ given by $y^{26}=\frac{(x^3+1)^2}{x^3}$. The covered curve $y^{13}=\frac{(x^3+1)^2}{x^3}$ of genus $18$ also provides a maximal curve. Moreover, in the Example \ref{Example2} we get a new maximal curve over $\mathbb{F}_{7^4}$ of genus $46$. These genera already appeared in \cite{DO2015} as the genus of a curve covered by the Hermitian curve.

These three examples of explicit maximal curves are new entries in manYPoints table \cite{MP2009} and rises a natural question, to decide if  these curves are or not covered by the Hermitian curve. 
\end{remark}

\begin{example}\label{Example8}
	Let $f(x)=x^4+b$ where $b \in \fq^*$ such that $ b^2\ne 1$. In the following tables, we list $q,m,b,s,g$ and $\# \cX(\fqs)$ where $m, s,\cX$ and $f$ satisfy the hypothesis of Theorem \ref{theoX1}.
	
	\begin{table}[h!]
		\centering
		\captionof*{table}{\bf{Meet record}}
		\begin{tabular}
{
|>{\centering\arraybackslash}p{0.5cm}|>{\centering\arraybackslash}p{0.5cm} |>{\centering\arraybackslash}p{0.5cm} |>{\centering\arraybackslash}p{0.5cm} |>{\centering\arraybackslash}p{0.5cm} |>{\centering\arraybackslash}p{1.6cm} |
}		
		\hline
			$q$ & $m$ & $b$ & $s$ & $g$ & $\# \cX(\fqs)$\\ \hline 
			$3^2$  & $10$ &  $\xi_{3^2}^2$ & $9$ & $36$ &   $730^\dagger$ \\    \hline
			$5^2$ &  $2$ &    $2$ & $1$ &  $4$ &   $826^\dagger$ \\    \hline
			$5^2$ & $13$ &    $2$ & $4$ & $48$ &  $3026^\dagger$ \\    \hline
			$11^2$&  $2$ & $\xi_{{11}^2}^{30}$ & $1$ &  $4$ & $15610^\dagger$ \\
			\hline           
			$17$ & $3$ & $4$ & $4$ & $8$ & $562^\dagger$ \\ \hline
			$17$ &  $9$ & $4$ & $4$ & $32$ &  $1378^\dagger$ \\    \hline 
		\end{tabular}
	\end{table}
\end{example}
Next we provide some more examples of curve with many points.

\begin{example}\label{Example9}
	In the following tables, we list $q,m,b,s,g$ and $\# \cX(\fqs)$ where $m, s,\cX$ and $f\in \fq[x]$ satisfy the hypothesis of Theorem \ref{theoX1}.
	
\begin{table}[h!]
	\centering
	\captionof*{table}{\bf{Meet record}}
	\begin{tabular}{|c|c|c|c|c|c|}\hline
		$q$ & $m$ & $f$ & $s$ & $g$ & $\# \cX(\fqs)$\\ \hline  
		$2$  & $3$ &    $x^3+x+1$ & $0$ & $4$ &   $15^\dagger$ \\     \hline
		$3$  & $4$ &  $x^2+2x+2$ & $0$ & $3$ &   $28^\dagger$ \\    \hline
		$3^2$ & $5$ & $x^4+x^2+2$ & $4$ & $16$ & $370^\dagger$  \\ \hline
		$7$  & $4$ &  $x^4+x^2+5$ & $0$ & $5$ &   $120^\dagger$ \\    \hline
		$7$  & $8$ &  $x^2+3x+3$ & $6$ & $9$ &   $176^\dagger$ \\    \hline
		$7^2$ & $5$ & $x^4+2x^2+3$ & $4$ & $16$ & $3970^\dagger$ \\ \hline
		$11$  & $6$ &  $x^2+3x+10$ & $0$ & $7$ &   $276^\dagger$ \\    \hline
		$11$  & $6$ &  $x^2+3x+10$ & $2$ & $9$ &   $320^\dagger$ \\    \hline
		$11$  & $12$ &  $x^2+3x+10$ & $0$ & $15$ &   $452^\dagger$ \\    \hline
		$11$  & $12$ &  $x^2+3x+10$ & $8$ & $19$ &   $540^\dagger$ \\    \hline
		$19$  & $10$ &  $x^2+6x+18$ & $0$ & $13$ &   $856^\dagger$ \\    \hline
		$19$  & $10$ &  $x^2+6x+18$ & $2$ & $17$ &   $1008^\dagger$ \\    \hline
		$19$  & $20$ &  $x^2+6x+18$ & $8$ & $35$ &   $1692^\dagger$ \\    \hline
		$19$  & $10$ &  $x^4+x^2+7$ & $9$ & $36$ &   $1730^\dagger$ \\    \hline
	\end{tabular}
\end{table}
\begin{table}[h!]
	\centering
	\captionof*{table}{\bf{New entry}}
	\begin{tabular}{|c|c|c|c|c|c|c|}\hline
	$q$ & $m$ & $f$ & $s$ & $g$ & $\# \cX(\fqs)$ & $\text{OLB}$\\ \hline
		$7^2$ &  $10$ & $x^2 + \xi_{7^2}x + \xi_{7^2}^{39}$ &  $3$  & $18$   & $3726$ &     $3649$ \\ \hline 
		$7^2$ & $10$ & $x^4+\xi_{7^2}x^2+\xi_{7^2}^{29}$ & $4$ & $25$ & $4272$ & $4134$ \\ \hline
		$7^2$ & $10$ & $x^4+2x^2+3$ & $4$ & $35$ & $5052$ & $4827$ \\ \hline
	$17$ &   $6$ & $x^4 + 6x^2 + 16$ &  $1$  & $20$   & $826$  &    $770$    \\ \hline
	$17$ &  $18$ & $x^3 + 14x + 2$ & $3$  & $23$  &  $892$ &     $842$ \\ \hline
	$19$ &  $10$ & $x^4 + 2x^2 + 16$ &  $4$  & $25$   & $1072$ &     $1033$ \\ \hline 
		\end{tabular}
\end{table}
\end{example}

\vspace{5cm}
Inspired by the previous constructions, we present some improvements obtained using Artin-Schreier extensions.

\begin{example}\label{Example10}
Let $\cX$ be the curve defined by the equation
$$
\cX:\quad y^q+y=\frac{f(x)f^*(x)}{x^s}
$$
where $f\in \fq[x]$ and $s\geq 0$ is an integer. We have the following improvements in manYPoints table in \cite{MP2009}.
\begin{table}[h!]
\centering
\captionof*{table}{\bf{New entry}}
\begin{tabular}{|c|c|c|c|c|c|}\hline
$q$ & $f(x)$ & $s$ & $g$ & $\# \cX(\fqs)$ & $\text{OLB}$ \\ \hline 
$ 7 $ & $ x^2 + 1$ & $2$ & $ 12 $ & $ 170 $  &  $ 165 $\\ \hline  
$ 11 $ & $ x^2 + 1$ & $2$ & $ 20 $ & $ 442 $  &  $ 430$\\ \hline 
$ 13 $ & $ x^2 + 1$ & $2$ & $ 24 $ & $ 626 $  &  $ 611 $\\ \hline 
\end{tabular}
\end{table}
\end{example}

\section{Curves over $\fqs$ from Section \ref{construction}: the case of $\epsilon=1$ and $\lambda=-1$}
\label{Section 5}
In this section, we consider the curve $\cX$ in (\ref{curveX}) with $\epsilon=1$ and $\lambda=-1$. As in Section \ref{Section 4}, we provide a lower bound for the number of $\fqs$-rational points on the curve $\cX$ when the polynomial $f \in \fq[x]$ satisfies certain conditions. We also provide some examples of curves with many points.

\begin{theorem}\label{theoX2}
Let $m \geq 2$ be a divisor of $q-1$, $f \in \fq[x]$ be a separable polynomial of degree $d$ satisfying $f(0)\neq 0$ and $(f, f^*)=1$, and $s$ be an integer $0\leq s <m$. Then the algebraic curve defined by the affine equation
\begin{equation}\label{curveX2}
\cX: \quad y^{m}=\frac{x^sf(x)}{f^*(x)}.
\end{equation}
has genus 
$$g=d(m-1)+1-(m,s).$$
Further if $(f, x^{q+1}-1)=1$, then the number of rational points $\#\cX(\fqs)$ over $\fqs$ satisfies
$$\#\cX(\fqs) \geq  2 N_f(\fqs) +m(q+1).$$
\end{theorem}
\begin{proof}
A direct application of Proposition \ref{prop1} gives the genus of the curve defined in (\ref{curveX2}). To obtain an expression for the number of rational points for this curve, we observe that for $\a \in \fqs^*$ with $f(\a)f^*(\a)\neq 0$, we have $\frac{\a^sf(\a)}{f^*(\a)}$ is a $m$-th power in $\fqs$ if and only if
$
\left(\frac{\a^sf(\a)}{f^*(\a)}\right)^{\frac{q^2-1}{m}}=1,$
which is equivalent to
$$((\a^sf(\a))^{q+1}-f^*(\a)^{q+1})\sum_{i=0}^{\frac{q-1}{m}-1}(\a^sf(\a))^{(q+1)i}f^*(\a)^{(q+1)\left(\frac{q-1}{m}-1-i\right)}=0.$$
Let
$$
h_1(x)=(x^sf(x))^{q+1}-f^*(x)^{q+1} \quad \text{ and} \quad
h_2(x)=\sum_{i=0}^{\frac{q-1}{m}-1}(x^sf(x))^{(q+1)i}f^*(x)^{(q+1)\left(\frac{q-1}{m}-1-i\right)}.
$$
Then $h_1$ and $h_2$ are co prime. In fact, if $\a$  a root of $h_1$ we have $(\a^sf(\a))^{q+1}=f^*(\a)^{(q+1)}$ and
\begin{align*}
h_2(\a)&=\sum_{i=0}^{\frac{q-1}{m}-1}(\a^sf(\a))^{(q+1)i}f^*(\a)^{(q+1)\left(\frac{q-1}{m}-1-i\right)}
=\sum_{i=0}^{\frac{q-1}{m}-1}f^*(\a)^{(q+1)i}f^*(\a)^{(q+1)\left(\frac{q-1}{m}-1-i\right)}\\
&=\sum_{i=0}^{\frac{q-1}{m}-1}f^*(\a)^{(q+1)\left(\frac{q-1}{m}-1\right)}
=\left(\frac{q-1}{m}\right)f^*(\a)^{(q+1)\left(\frac{q-1}{m}-1\right)}\neq 0.
\end{align*}
Also, since $(h_1, ff^*)=(h_2, ff^*)=1$, we obtain
$$\#\{ \a \in \fqs \mid f(\a)f^*(\a)\neq 0 \text{ and }\frac{\a^sf(\a)}{f^*(\a)} \text{ is a } m\text{-th power in} \fqs^*\}=N_{h_1}(\fqs)+N_{h_2}(\fqs).$$
On the other hand, from Remark \ref{remark1}, we know that each root in $\fqs$ of the polynomial $ff^*$ gives one rational point on the curve.
Thus 
\begin{equation}\label{NRPX2}
\# \cX (\fqs)\geq 2N_f(\fqs)+m(N_{h_1}(\fqs)+N_{h_2}(\fqs)).
\end{equation}
Next we assume $(f, x^{q+1}-1)=1$. Then for $\beta\in \fqs$ such that $\beta^{q+1}=1$, we have 
\begin{align*}
h_1(\beta)&=(\beta^sf(\beta))^{q+1}-f^*(\beta)^{q+1}\\
&=\beta^{s(q+1)}f(\beta)^{q+1}-\beta^{d(q+1)}f(\beta)^{q+1}\\
&=0.
\end{align*}
Therefore $N_{h_1}(\fqs)\geq q+1$. Hence the assertion follows from \eqref{NRPX2}.
\end{proof}

From the constructions given in Theorem \ref{theoX2}, we obtain the following examples of curves with many points.

\begin{example}\label{Example3}
Let $f(x)=x+b \in \fq[x]$ such that $b\neq 0, b^2\neq 1$ and $m, s, \cX$ be as defined in Theorem \ref{theoX2}. Then $ (f, f^*)=(f, x^{q+1}-1)=1$ and the curve $\cX$ has genus $g=m-(m,s)$. We obtain the following tables of curves with many points.\\

\begin{minipage}[c]{0.5\textwidth}
\centering
\captionof*{table}{\bf{New entry}}
\begin{tabular}
{
|>{\centering\arraybackslash}p{0.5cm}|>{\centering\arraybackslash}p{0.5cm} |>{\centering\arraybackslash}p{0.5cm} |>{\centering\arraybackslash}p{0.5cm} |>{\centering\arraybackslash}p{0.5cm} |>{\centering\arraybackslash}p{1.5cm} |
>{\centering\arraybackslash}p{1.1cm} |
}
\hline
$q$ & $m$ & $b$ & $s$ & $g$ & $\# \cX(\fqs)$ & $\text{OLB}$\\ \hline 
$ 5^2 $ & $ 24 $ & $ \xi_{5^2}^3 $ & $ 8 $ & $ 16 $ & $ 1202 $ & $ 1191 $ \\ \hline
$ 5^2 $ & $ 24 $ & $ 2 $ & $ 4 $ & $ 20 $ & $ 1450 $ & $ 1333 $ \\ \hline
$ 5^2 $ & $ 24 $ & $ 2 $ & $ 9 $ & $ 21 $ & $ 1400 $ & $ 1368 $ \\ \hline
$ 7^2 $ & $ 16 $ & $ \xi_{7^2}^3 $ & $6$ & $ 14 $ & $ 3558 $ & $ 3372 $\\ \hline
$ 7^2 $ & $ 16 $ & $ \xi_{7^2}^5 $ & $7$ & $ 15 $ & $ 3684 $ & $ 3441 $\\ \hline
$ 7^2 $ & $ 24 $ & $ \xi_{7^2}^{13} $ & $5$ & $ 23 $ & $ 4276 $ & $ 3995 $\\ \hline
$ 11^2 $ & $ 15 $ & $ \xi_{11^2}^2 $ & $ 4 $ & $ 14 $ & $ 17674 $ & $ 17037 $\\ \hline
$ 11^2 $ & $ 24 $ & $ \xi_{11^2}^{25} $ & $ 8 $ & $ 16 $ & $ 18050 $ & $ 17379 $\\ \hline
$ 11^2 $ & $ 24 $ & $ 5 $ & $ 3 $ & $ 21 $ & $ 18968 $ & $ 18235 $\\ \hline
$ 11^2 $ & $ 24 $ & $ \xi_{11^2}^{21} $ & $ 7 $ & $ 23 $ & $ 19204 $ & $ 18577 $\\ \hline
$ 11^2 $ & $ 30 $ & $ \xi_{11^2}^{9} $ & $ 3 $ & $ 27 $ & $ 19988 $ & $ 19262 $\\ \hline
$ 11^2 $ & $ 30 $ & $ \xi_{11^2}^{2} $ & $ 4 $ & $ 28 $ & $ 20106 $ & $ 19433 $\\ \hline
$ 11^2 $ & $ 40 $ & $ \xi_{11^2}^{7} $ & $ 4 $ & $ 36 $ & $ 20962 $ & $ 20802 $\\ \hline
$ 11^2 $ & $ 40 $ & $ \xi_{11^2}^{13} $ & $ 6 $ & $ 38 $ & $ 22246 $ & $ 21144 $\\ \hline
$ 13^2 $ & $ 12 $ & $ \xi_{13^2}^{23} $ & $5$ & $ 11 $ & $ 31972 $ & $ 31191 $\\ \hline
$ 13^2 $ & $ 14 $ & $ \xi_{13^2}^{10} $ & $9$ & $ 13 $ & $ 32260 $ & $ 31669 $\\ \hline
$ 13^2 $ & $ 21 $ & $ \xi_{13^2}^2 $ & $5$ & $ 20 $ & $ 34318 $ & $ 33342 $\\ \hline
$ 13^2 $ & $ 24 $ & $ \xi_{13^2}^{23} $ & $5$ & $ 23 $ & $ 35428 $ & $ 34059 $\\ \hline
$ 13^2$ & $ 28 $ & $ \xi_{13^2}^5 $ & $9$ & $ 27 $ & $ 35452 $ & $ 35015 $\\ \hline
$ 13^2 $ & $ 42 $ & $ \xi_{13^2}^{11}$ & $7$ & $ 35 $ & $ 37550 $ & $ 36927 $\\ \hline

\end{tabular}
\end{minipage}
\hfill
\begin{minipage}[c]{0.5\textwidth}
\centering
\captionof*{table}{\bf{New entry}}
\begin{tabular}
{
|>{\centering\arraybackslash}p{0.5cm}|>{\centering\arraybackslash}p{0.5cm} |>{\centering\arraybackslash}p{0.5cm} |>{\centering\arraybackslash}p{0.5cm} |>{\centering\arraybackslash}p{0.5cm} |>{\centering\arraybackslash}p{1.5cm} |
>{\centering\arraybackslash}p{1.1cm} |
}
\hline
$q$ & $m$ & $b$ & $s$ & $g$ & $\# \cX(\fqs)$ & $\text{OLB}$\\ \hline 
$ 17^2 $ & $ 12 $ & $ \xi_{17^2}^4 $ & $9$ & $ 9 $ & $ 87938 $ & $ 87200 $\\ \hline
$ 17^2 $ & $ 12 $ & $ \xi_{17^2}^6 $ & $5$ & $ 11 $ & $ 88828 $ & $ 88017 $\\ \hline
$ 17^2 $ & $ 16 $ & $ \xi_{17^2}^7 $ & $5$ & $ 15 $ & $ 91044 $ & $ 89652 $\\ \hline
$ 17^2 $ & $ 24 $ & $\xi_{17^2}^4 $ & $5$ & $ 23 $ & $ 94996 $ & $ 92922 $\\ \hline
$ 17^2 $ & $ 32 $ & $ \xi_{17^2}^7 $ & $5$ & $ 31 $ & $ 97604 $ & $ 96191 $\\ \hline
$ 17^2 $ & $ 48 $ & $\xi_{17^2}^4 $ & $5$ & $ 47 $ & $ 105124 $ & $ 102731 $\\ \hline
\end{tabular}
\vspace{0.3cm}
\centering
\captionof*{table}{\bf{New record}}
\begin{tabular}
{
|>{\centering\arraybackslash}p{0.5cm}|>{\centering\arraybackslash}p{0.5cm} |>{\centering\arraybackslash}p{0.5cm} |>{\centering\arraybackslash}p{0.5cm} |>{\centering\arraybackslash}p{0.5cm} |>{\centering\arraybackslash}p{1.5cm} |
>{\centering\arraybackslash}p{1.1cm} |
}
\hline
$q$ & $m$ & $b$ & $s$ & $g$ & $\# \cX(\fqs)$ & $\text{OLB}$\\ \hline 
$ 17 $ & $ 16 $ & $ 3 $ & $9$ & $ 15 $ & $ 708 $ & $ 692 $\\ \hline
$ 11^2 $ & $ 15 $ & $ \xi_{11^2}^{26} $ & $ 5 $ & $ 10 $ & $ 16952 $ & $ 16942 $\\ \hline
$ 19 $ & $ 18 $ & $ 2 $ & $3$ & $ 15 $ & $ 866 $ & $ 782 $\\ \hline
\end{tabular}
\vspace{0.3cm}
\centering
\captionof*{table}{\bf{Meet record}}
\begin{tabular}
{
|>{\centering\arraybackslash}p{0.7cm}|>{\centering\arraybackslash}p{0.7cm} |>{\centering\arraybackslash}p{0.7cm} |>{\centering\arraybackslash}p{0.7cm} |>{\centering\arraybackslash}p{0.7cm} |>{\centering\arraybackslash}p{1.9cm} |
}
\hline
$q$ & $m$ & $b$ & $s$ & $g$ & $\# \cX(\fqs)$\\ \hline 
$ 7 $ & $ 6 $ & $ 2 $ & $4$ & $ 4 $ & $ 102 $\\ \hline
$ 11^2 $ & $ 8 $ & $ \xi_{11^2}^{30} $ & $4$ & $ 4 $ & $ 15610^\dagger $\\ \hline
$ 11^2 $ & $ 8 $ & $ \xi_{11^2}^{37} $ & $3$ & $ 7 $ & $ 16308 $\\ \hline
$ 13 $ & $ 12 $ & $ 2 $ & $4$ & $ 8 $ & $ 362 $\\ \hline
\end{tabular}
\end{minipage}
\end{example}

\begin{example}\label{Example4}
Let $f(x)=x^2+b \in \fq[x]$ such that $b\neq 0, b^2\neq 1$, and $m, s, \cX$ be as defined in Theorem \ref{theoX2}. Then $(f, f^*)=1$ and the curve $\cX$ has genus $g=2m-1-(m,s)$. We have the following tables.\\

\begin{minipage}[c]{0.5\textwidth}
\centering
\captionof*{table}{\bf{New entry}}
\begin{tabular}
{
|>{\centering\arraybackslash}p{0.5cm}|>{\centering\arraybackslash}p{0.5cm} |>{\centering\arraybackslash}p{0.5cm} |>{\centering\arraybackslash}p{0.5cm} |>{\centering\arraybackslash}p{0.5cm} |>{\centering\arraybackslash}p{1.5cm} |
>{\centering\arraybackslash}p{1.1cm} |
}
\hline
$q$ & $m$ & $b$ & $s$ & $g$ & $\# \cX(\fqs)$ & $\text{OLB}$\\ \hline 
$ 5^2 $ & $ 8 $ & $\xi_{5^2}$ & $ 2 $ & $ 13 $ & $ 1128 $ & $ 1085 $\\ \hline 
$ 5^2 $ & $ 24 $ & $ \xi_{5^2} $ & $ 10 $ & $ 45 $ & $ 2216 $ & $ 2216 $\\ \hline  
$ 7^2 $ & $ 12 $ & $\xi_{7^2}^6$ & $ 2 $ & $ 21 $ & $ 4040 $ & $ 3857 $\\ \hline  
$ 7^2 $ & $ 16 $ & $ \xi_{7^2}^{13}$ & $ 2 $ & $ 29 $ & $ 4552 $ & $ 4411 $\\ \hline 
$ 7^2 $ & $ 24 $ & $ \xi_{7^2}^5$ & $ 6 $ & $ 41 $ & $ 5380 $ & $ 5243 $\\ \hline 
$ 7^2 $ & $ 24 $ & $ \xi_{7^2}^3$ & $ 4 $ & $ 43 $ & $ 5476 $ & $ 5381 $\\ \hline 
$ 7^2 $ & $ 24 $ & $ \xi_{7^2}^{10}$ & $ 10 $ & $ 45 $ & $ 5672 $ & $ 5520 $\\ \hline  
$ 11 $ & $ 10 $ & $ 3$ & $ 2 $ & $ 17 $ & $ 408 $ & $ 386 $\\ \hline
$11^2$ & $8$ & $\xi_{11^2}^{18}$ & $4$ & $11$ & $16940$ & $16524$ \\ \hline
$ 11^2 $ & $ 8 $ & $ \xi_{11^2}^7$ & $ 2 $ & $ 13 $ & $ 17480 $ & $ 16866 $\\ \hline 
$ 11^2 $ & $ 10 $ & $ \xi_{11^2}^{19}$ & $ 2 $ & $ 17 $ & $ 18128 $ & $ 17551 $\\ \hline 
$11^2$ & $12$ & $\xi_{11^2}^{43}$ & $4$ & $19$ & $18436$ & $17893$ \\ \hline
$ 11^2 $ & $ 15 $ & $ \xi_{11^2}^{2} $ & $ 5 $ & $ 24 $ & $ 18964 $ & $ 18748 $\\ \hline
$ 11^2 $ & $ 15 $ & $ \xi_{11^2}^{31}$ & $ 3 $ & $ 26 $ & $ 19564 $ & $ 19091 $\\ \hline 
$ 11^2 $ & $ 30 $ & $\xi_{11^2}$ & $ 0 $ & $ 29 $ & $ 19684 $ & $ 19604 $\\ \hline 
$ 11^2 $ & $ 20 $ & $\xi_{11^2}^{26}$ & $ 5 $ & $ 34 $ & $ 20644 $ & $ 20460 $\\ \hline 
$ 11^2 $ & $ 20 $ & $ \xi_{11^2}^{37} $ & $ 8 $ & $ 35 $ & $ 20964 $ & $ 20631 $\\ \hline
$ 11^2 $ & $ 20 $ & $\xi_{11^2}^{13}$ & $ 2 $ & $ 37 $ & $ 21528 $ & $ 20973 $\\ \hline 
$11^2$ & $24$ & $\xi_{11^2}^{25}$ & $4$ & $43$ & $22372$ & $22000$ \\ \hline
$ 11^2 $ & $ 24 $ & $\xi_{11^2}^{21}$ & $ 2 $ & $ 45 $ & $ 22472 $ & $ 22342 $\\ \hline 
\end{tabular}
\end{minipage}
\hfill
\begin{minipage}[c]{0.5\textwidth}
\centering
\captionof*{table}{\bf{New entry}}
\begin{tabular}
{
|>{\centering\arraybackslash}p{0.5cm}|>{\centering\arraybackslash}p{0.5cm} |>{\centering\arraybackslash}p{0.5cm} |>{\centering\arraybackslash}p{0.5cm} |>{\centering\arraybackslash}p{0.5cm} |>{\centering\arraybackslash}p{1.5cm} |
>{\centering\arraybackslash}p{1.1cm} |
}
\hline
$q$ & $m$ & $b$ & $s$ & $g$ & $\# \cX(\fqs)$ & $\text{OLB}$\\ \hline 
$ 13 $ & $ 12 $ & $2$ & $ 8 $ & $ 19 $ & $ 532 $ & $ 519 $\\ \hline 
$13^2$ & $12$ & $\xi_{13^2}^{5}$ & $4$ & $19$ & $33748$ & $33103$ \\ \hline
$ 13^2 $ & $ 12 $ & $\xi_{13^2}^{44}$ & $ 1 $ & $ 22 $ & $ 34374 $ & $ 33820 $\\ \hline
$13^2$ & $14$ & $\xi_{13^2}^{8}$ & $2$ & $25$ & $35400$ & $34537$\\ \hline
$ 13^2 $ & $ 21 $ & $ 8$ & $ 3 $ & $ 38 $ & $ 38314 $ & $ 37644 $\\ \hline
$ 13^2 $ & $ 24 $ & $\xi_{13^2}^{23}$ & $ 10 $ & $ 45 $ & $ 40136 $ & $ 39317 $\\ \hline
$ 17 $ & $ 16 $ & $ 2$ & $ 4 $ & $ 27 $ & $ 972 $ & $ 939 $\\ \hline 
$ 17^2 $ & $ 8 $ & $\xi_{17^2}$ & $ 6 $ & $ 13 $ & $ 89224 $ & $ 88835 $\\ \hline
$ 17^2 $ & $ 18 $ & $\xi_{17^2}$ & $ 7 $ & $ 34 $ & $ 97494 $ & $ 97418 $\\ \hline
$ 19 $ & $ 18 $ & $ 2$ & $ 6 $ & $ 29 $ & $ 1156 $ & $ 1141 $\\ \hline 
$ 19^2 $ & $ 6 $ & $\xi_{19^2}$ & $ 1 $ & $ 10 $ & $ 136782 $ & $ 135427 $\\ \hline
$ 19^2 $ & $ 12 $ & $\xi_{19^2}$ & $ 0 $ & $ 11 $ & $ 136612 $ & $ 135937 $\\ \hline
$19^2$ & $9$ & $\xi_{19^2}$ & $3$ & $14$ & $138208$ & $137469$\\ \hline
$19^2$ & $9$ & $\xi_{19^2}$ & $2$ & $16$ & $139650$ & $138490$\\ \hline
$ 19^2 $ & $ 24 $ & $\xi_{19^2}$ & $ 0 $ & $ 23 $ & $ 142372 $ & $ 142064 $\\ \hline
\end{tabular}
\vspace{0.7cm}
\centering
\captionof*{table}{\bf{New record}}
\begin{tabular}
{
|>{\centering\arraybackslash}p{0.5cm}|>{\centering\arraybackslash}p{0.5cm} |>{\centering\arraybackslash}p{0.5cm} |>{\centering\arraybackslash}p{0.5cm} |>{\centering\arraybackslash}p{0.5cm} |>{\centering\arraybackslash}p{1.5cm} |
>{\centering\arraybackslash}p{1.1cm} |
}
\hline
$q$ & $m$ & $b$ & $s$ & $g$ & $\# \cX(\fqs)$ & $\text{OLB}$\\ \hline 
$ 17 $ & $ 8 $ & $4$ & $ 2 $ & $ 13 $ & $ 648 $ & $ 612 $\\ \hline 
\end{tabular}
\end{minipage}
\end{example}

\begin{example}\label{Example5}
Let $f(x)=x^3+b \in \fq[x]$ such that $b\neq 0, b^2\neq 1$, and $m, s, \cX$ be as defined in Theorem \ref{theoX2}. Then $(f, f^*)=1$ and the curve $\cX$ has genus $g=3m-2-(m,s)$. We have the following tables.\\
\begin{minipage}[c]{0.5\textwidth}
\centering
\captionof*{table}{\bf{New entry}}
\begin{tabular}
{
|>{\centering\arraybackslash}p{0.5cm}|>{\centering\arraybackslash}p{0.5cm} |>{\centering\arraybackslash}p{0.5cm} |>{\centering\arraybackslash}p{0.5cm} |>{\centering\arraybackslash}p{0.5cm} |>{\centering\arraybackslash}p{1.5cm} |
>{\centering\arraybackslash}p{1.1cm} |
}
\hline
$q$ & $m$ & $b$ & $s$ & $g$ & $\# \cX(\fqs)$ & $\text{OLB}$\\ \hline 
$ 7^2 $ & $ 12 $ & $\xi_{7^2}^2 $ & $ 3 $ & $ 31 $ & $ 4680 $ & $ 4550 $\\ \hline
$ 11^2 $ & $ 8 $ & $\xi_{11^2}^{15} $ & $ 4 $ & $ 18 $ & $ 18486 $ & $ 17722 $\\ \hline
$ 11^2 $ & $ 12 $ & $ \xi_{11^2}^2 $ & $ 0 $ & $ 22 $ & $ 19080 $ & $ 18406 $\\ \hline
$ 11^2 $ & $ 12 $ & $ \xi_{11^2}^9 $ & $ 3 $ & $ 31 $ & $ 20820 $ & $ 19946 $\\ \hline
$ 11^2 $ & $ 15 $ & $ \xi_{11^2}^{30} $ & $ 6 $ & $ 40 $ & $ 22242 $ & $ 21486 $\\ \hline
$ 11^2 $ & $ 24 $ & $ \xi_{11^2}^2 $ & $ 0 $ & $ 46 $ & $ 22608 $ & $ 22513 $\\ \hline
$ 13^2 $ & $ 7 $ & $ \xi_{13^2}^4 $ & $ 3 $ & $ 18 $ & $ 33980 $ & $ 32864 $\\ \hline
$ 13^2 $ & $ 12 $ & $ \xi_{13^2}^{23} $ & $ 3 $ & $ 31 $ & $ 36792 $ & $ 35971 $\\ \hline
$ 13^2 $ & $ 14 $ & $ \xi_{13^2}^{21} $ & $ 7 $ & $ 33 $ & $ 37568 $ & $ 36449 $\\ \hline
$ 13^2 $ & $ 14 $ & $ \xi_{13^2}^3 $ & $ 3 $ & $ 39 $ & $ 38732 $ & $ 37883 $\\ \hline
$ 19 $ & $ 9 $ & $ 4 $ & $ 6 $ & $ 22 $ & $ 972 $ & $ 953 $ \\ \hline
\end{tabular}
\end{minipage}
\hfill
\begin{minipage}[c]{0.5\textwidth}
\centering
\captionof*{table}{\bf{New record}}
\begin{tabular}
{
|>{\centering\arraybackslash}p{0.5cm}|>{\centering\arraybackslash}p{0.5cm} |>{\centering\arraybackslash}p{0.5cm} |>{\centering\arraybackslash}p{0.5cm} |>{\centering\arraybackslash}p{0.5cm} |>{\centering\arraybackslash}p{1.5cm} |
>{\centering\arraybackslash}p{1.1cm} |
}
\hline
$q$ & $m$ & $b$ & $s$ & $g$ & $\# \cX(\fqs)$ & $\text{OLB}$\\ \hline 
$ 11^2 $ & $ 5 $ & $\xi_{11^2}^6 $ & $ 0 $ & $ 8 $ & $ 16566 $ & $ 16546 $\\ \hline 
\end{tabular}
\end{minipage}
\end{example}

\begin{example}\label{Example6}
Let $f(x)=x^4+b \in \fq[x]$ such that $b\neq 0, b^2\neq 1$, and $m, s, \cX$ be as defined in Theorem \ref{theoX2}. Then $(f, f^*)=1$ and the curve $\cX$ has genus $g=4m-3-(m,s)$. We have the following tables of curves with many points.\\
\begin{table}[h!]
\centering
\captionof*{table}{\bf{New entry}}
\begin{tabular}
{
|>{\centering\arraybackslash}p{0.5cm}|>{\centering\arraybackslash}p{0.5cm} |>{\centering\arraybackslash}p{0.5cm} |>{\centering\arraybackslash}p{0.5cm} |>{\centering\arraybackslash}p{0.5cm} |>{\centering\arraybackslash}p{1.5cm} |
>{\centering\arraybackslash}p{1.1cm} |
}
\hline
$q$ & $m$ & $b$ & $s$ & $g$ & $\# \cX(\fqs)$ & $\text{OLB}$\\ \hline 
$ 7^2 $ & $ 8 $ & $ 3 $ & $ 5 $ & $ 28 $ & $ 4522 $ & $ 4342 $\\ \hline
$ 11^2 $ & $ 6 $ & $ \xi_{11^2}^{20} $ & $ 0 $ & $ 15 $ & $ 17672 $ & $ 17208 $\\ \hline
$ 11^2 $ & $ 8 $ & $ \xi_{11^2}^7 $ & $ 4 $ & $ 25 $ & $ 19456 $ & $ 18919 $\\ \hline
$ 11^2 $ & $ 12 $ & $ \xi_{11^2}^{14} $ & $ 6 $ & $ 39 $ & $ 22184 $ & $ 21315 $\\ \hline
$ 13 $ & $ 6 $ & $ 2 $ & $ 1 $ & $ 20 $ & $ 554 $ & $ 537 $\\ \hline
$ 13^2 $ & $ 6 $ & $ \xi_{13^2}^{22} $ & $ 0 $ & $ 15 $ & $ 32216 $ & $ 32147 $\\ \hline 
$ 13^2 $ & $ 8 $ & $ \xi_{13^2}^4 $ & $ 0 $ & $ 21 $ & $ 34584 $ & $ 33581 $\\ \hline 
$ 13^2 $ & $ 12 $ & $ \xi_{13^2}^{23} $ & $ 8 $ & $ 41 $ & $ 38784 $ & $ 38361 $\\ \hline
$ 17^2 $ & $ 6 $ & $ \xi_{17^2} $ & $ 5 $ & $ 20 $ & $ 91826 $ & $ 91696 $\\ \hline
\end{tabular}
\end{table}
\end{example}
We finish this section by giving some additional improvements in manYPoints table in \cite{MP2009}.
\begin{example}\label{Example7}
Let $f\in \fq[x]$ and $m, s, \cX$ be as defined in Theorem \ref{theoX2}. We have the following curves with many points. 

\begin{table}[h!]
\centering
\captionof*{table}{\bf{New entry}}
\begin{tabular}{|c|c|c|c|c|c|c|}\hline
$q$ & $m$ & $f$ & $s$ & $g$ & $\# \cX(\fqs)$ & $\text{OLB}$\\ \hline
$ 5^2 $ & $ 4 $ & $ x^4+\xi_{5^2}x^2+\xi_{5^2}^{7} $ & $ 4 $ & $ 9 $ & $ 984 $ & $ 944 $\\ \hline 
$ 5^2 $ & $ 8 $ & $ x^2+2x+\xi_{5^2}^{7} $ & $ 4 $ & $ 11 $ & $ 1092 $ & $ 1014 $\\ \hline
$ 5^2 $ & $ 8 $ & $ x^2+\xi_{5^2}^3x+2 $ & $ 7 $ & $ 14 $ & $ 1206 $ & $ 1120 $\\ \hline
$ 5^2 $ & $ 6 $ & $ x^4+x^2+\xi_{5^2}^{14} $ & $ 6 $ & $ 15 $ & $ 1160 $ & $ 1156 $\\ \hline
$ 5^2 $ & $ 12 $ & $ x^2+\xi_{5^2}^2x+\xi_{5^2}^{8} $ & $ 6 $ & $ 17 $ & $ 1252 $ & $ 1227 $\\ \hline
$ 5^2 $ & $ 6 $ & $ x^4+x^2+\xi_{5^2}^{7} $ & $ 4 $ & $ 19 $ & $ 1308 $ & $ 1297 $\\ \hline
$ 7^2 $ & $ 6 $ & $ x^4+\xi_{7^2}^{44}x^2+5 $ & $ 4 $ & $ 19 $ & $ 3780 $ & $ 3718 $\\ \hline
$ 7^2 $ & $ 8 $ & $ x^6+\xi_{7^2}^6 $ & $ 3 $ & $ 42 $ & $ 5486 $ & $ 5312 $\\ \hline
$ 11^2 $ & $ 5 $ & $ x^8+\xi_{11^2}^{14} $ & $ 2 $ & $ 32 $ & $ 20482 $ & $ 20117 $\\ \hline
$ 11^2 $ & $ 5 $ & $ x^{12}+\xi_{11^2}^{4} $ & $ 0 $ & $ 44 $ & $ 22800 $ & $ 22171 $\\ \hline
$ 13^2 $ & $ 6 $ & $ x^7+2 $ & $ 0 $ & $ 30 $ & $ 35966 $ & $ 35732 $\\ \hline
\end{tabular}
\end{table}
\end{example}

\section{Curves with many points from fibre product of curves} \label{fibreproduct}
In this section, we construct new curves with many rational points by considering the fibre product of the curves constructed in Sections \ref{Section 4} and \ref{Section 5}. To provide a lower bound for the number of $\fqs$-rational points for these new constructions, we use a generalization of the Remark \ref{remark1} given in \cite[Theorem 4]{OT2014} for fibre products of Kummer extensions.

\begin{theorem}\label{theoX1X1}
	Let $i\in\{1, 2\}$. Let $m_i\geq 2$ be a divisor of $q+1$, $s_i$ be an integer with $0 \leq s_i < m_i$ and $f_i$ be a separable polynomial in $\fq [x]$ of degree $d_i$ satisfying $f_i(0)\neq 0$ and $(f_i, f_i^*)=(f_1f_1^*, f_2f_2^*)=1$. Then the curve $\cX$ defined by the affine equations
	\begin{equation}\label{curveX1X1}
		\cX: \left\{
		\begin{array}{l}
			y_2^{m_2}=\frac{f_2(x)f_2^*(x)}{x^{s_2}}\\
			y_1^{m_1}=\frac{f_1(x)f_1^*(x)}{x^{s_1}}
		\end{array} \right. 
	\end{equation}
	has genus 
	$$
	g=m_1m_2(d_1+d_2)-d_1m_2-d_2m_1+1-\frac{\kappa+(m_1m_2, m_2(2d_1-s_1), m_1(2d_2-s_2))}{2}
	$$ 
	where $\kappa=(m_1m_2, s_1m_2, s_2m_1)$. Further, if $\fqs$ is the full constant field of $\fqs(\cX)$, $[\fqs(\cX), \fqs(x)]=m_1m_2$ and $(f_i, x^{q+1}-1)=1$ for $i\in\{1, 2\}$, then the number of $\fqs$-rational points $\#\cX (\fqs)$ of the curve $\cX$ satisfies 
	\begin{align*}
		\#\cX (\fqs)&\geq m_1m_2((q+1, 2(d_1-s_1), 2(d_2-s_2))+q-3-2N_{f_1f_2}(\fq^*))\\
		&\hspace{5.7cm} +2m_2N_{f_1}(\fq^*)+2m_1N_{f_2}(\fq^*).
	\end{align*}
	In particular, for $s_1=d_1$ and $s_2=d_2$, we have
	$$
	\#\cX (\fqs)\geq 2m_1m_2(q-1-N_{f_1f_2}(\fq^*))+2m_2N_{f_1}(\fq^*)+2m_1N_{f_2}(\fq^*).
	$$
\end{theorem}
\begin{proof}
	We start by computing the genus of the function field $K(x, y_1, y_2)$. By Theorem \ref{theoX1},  we have $g(K(x, y_1))=(2m_1d_1-2(d_1-1)-(m_1, s_1)-(m_1, 2d_1-s_1))/2$.  Also, for the roots $\gamma_1, \dots, \gamma_{d_1}$  of $f_1$ in $K$, we have the following ramification indices $e(P)$ in the extension $K(x,y_1)/K(x)$.
	$$ 
	e(P)= \left\{
	\begin{array}{ll}
		m_1/ (m_1, s_1), & \mbox{if }P \text{ is over } P_0,\\
		m_1,  & \mbox{if }P \text{ is over } P_{\gamma_i} \text{ or }P_{\gamma_i^{-1}},\\
		m_1/ (m_1, 2d_1-s_1),  & \mbox{if }P \text{ is over } P_{\infty},\\
		1,   & \mbox{otherwise}.
	\end{array} \right .
	$$ 
	Now we show that the extension $K(x,y_1, y_2)/K(x, y_1)$ is a Kummer extension. Let $\a_1, \dots, \a_{d_2} \in K$ be the roots of $f_2$. The principal divisors of the function $x^{-s_2}f_2(x)f_2^*(x)$ in $K(x)$ is given by  
	$$
	(x^{-s_2}f_2(x)f_2^*(x))_{K(x)}=\sum_{i=1}^{d_2}(P_{\alpha_i}+P_{\a_i^{-1}})-s_2P_0-(2d_2-s_2)P_{\infty},
	$$
	and consequently
	\begin{align*}
		(x^{-s_2}f_2(x)f_2^*(x))_{K(x, y_1)}&=\sum_{j=1}^{m_1}\sum_{i=1}^{d_2}(Q_{\alpha_i, j}+Q_{\a_i^{-1}, j})-\frac{s_2m_1}{(m_1, s_1)}\sum_{i=1}^{(m_1, s_1)}Q_{0, i}\\
		&\quad -\frac{m_1(2d_2-s_2)}{(m_1, 2d_1-s_1)}\sum_{i=1}^{(m_1, 2d_1-s_1)}Q_{\infty, i}
	\end{align*}
	where $Q_{\a_i, j}$, $Q_{\a_i^{-1}, j}$, $Q_{0, i}$ and $Q_{\infty, i}$ are the extensions in $K(x, y_1)$ of the places $P_{\a_i}$, $P_{\a_i^{-1}}$, $P_0$ and $P_{\infty}$ respectively. Thus the ramification indices in the extension $K(x, y_1, y_2)/K(x, y_1)$ are given by
	$$ 
	e(R)= \left\{
	\begin{array}{ll}
		m_2(m_1, s_1)/\kappa, \hspace{5mm} & \mbox{if }R \text{ is over } Q_{0, i},\\
		m_2(m_1, 2d_1-s_1)/(m_1m_2, m_1(2d_2-s_2), m_2(2d_1-s_1)), \hspace{5mm} & \mbox{if }R \text{ is over } Q_{\infty, i},\\
		m_2,  & \mbox{if }R \text{ is over } Q_{\alpha_i, j} \text{ or }Q_{\a_i^{-1}, j},\\
		1,  \hspace{25mm} & \mbox{otherwise}.
	\end{array} \right. 
	$$ 
	We conclude that the equations (\ref{curveX1X1}) define an absolutely irreducible curve. Its genus follows from the Riemann-Hurwitz formula applied to $K(x, y_1, y_2)/K(x, y_1)$.
	
	Next we provide a lower bound for the number of $\fqs$-rational points. From \cite[Theorem 4]{OT2014}, it follows that
	\begin{itemize}
		\item for $\a\in \fqs^*$ such that $f_1f_1^*f_2f_2^*(\a)\neq 0$, the curve $\cX$ has $m_1m_2$ points with  coordinate $x=\a$  if and only if $\frac{f_i(\a)f_i^*(\a)}{\a^{s_i}}$ is a $m_i$-power in $\fqs^*$ for $i\in\{1, 2\}$,
		\item for $\a\in \fqs^*$ such that $f_1f_1^*(\a)= 0$, the curve $\cX$ has $m_2$ points with  coordinate $x=\a$  if and only if $\frac{f_2(\a)f_2^*(\a)}{\a^{s_2}}$ is a $m_2$-power in $\fqs^*$,
		\item for $\a\in \fqs^*$ such that $f_2f_2^*(\a)= 0$, the curve $\cX$ has $m_1$ points with  coordinate $x=\a$  if and only if $\frac{f_1(\a)f_1^*(\a)}{\a^{s_1}}$ is a $m_1$-power in $\fqs^*$.
	\end{itemize}
	From the proof of Theorem \ref{theoX1}, for $i\in \{1, 2\}$, we have
	$$
	\left(\frac{f_i(\a)f_i^*(\a)}{\a^{s_i}}\right)^{\frac{q^2-1}{m_i}}=1\Leftrightarrow \a \text{ is a root of } h_{i, 1}h_{i, 2}
	$$
	where
	$$
	h_{i, 1}(x)=(f_i(x)f_i^*(x))^{q-1}-x^{s_i(q-1)}\quad \text{and}\quad h_{i, 2}(x)=\sum_{j=0}^{\frac{q+1}{m_i}-1}(f_i(x)f_i^*(x))^{(q-1)j}x^{s_i(q-1)\left(\frac{q+1}{m_i}-1-j\right)}.
	$$
	From the proof of Theorem  \ref{theoX1}, we also have that if $\beta \in \fqs$ satisfies $\beta^{(q+1, 2(d_1-s_1), 2(d_2-s_2))}=1$, then $h_{i, 1}(\beta)=0$. Further, for $i\in \{1, 2\}$, if $\beta\in \fq^*$ and $f_i(\beta)f_i^*(\beta)\neq 0$, then $h_{i, 1}(\beta)=0$. Hence 
	\begin{align*}
		\#\cX (\fqs)&\geq m_1m_2((q+1, 2(d_1-s_1), 2(d_2-s_2))+q-3-2N_{f_1f_2}(\fq^*))\\
		&\hspace{5.7cm} +2m_2N_{f_1}(\fq^*)+2m_1N_{f_2}(\fq^*).
	\end{align*}
\end{proof}
\begin{example}\label{Example12}
	For polynomials $f_1, f_2\in \fq[x]$ satisfying the conditions of Theorem \ref{theoX1X1} and the curve $\cX$ as defined in (\ref{curveX1X1}), we have the following table.
	
	\begin{table}[h!]
		\captionof*{table}{\bf{New entry}}
		\centering
		\begin{tabular}{|c|c|c|c|c|c|c|c|c|c|c|c|c|c|}\hline 
			$q$ & $m_1$ & $m_2$ & $s_1$ & $s_2$ & $f_1$ & $f_2$ & $g$ & $\# \cX(\fqs)$ & $\text{OLB}$ \\ \hline 
			$19$ & $2$ & $4$ & $4$ & $4$ & $x^4+2$ & $x^4+7$ & $33$ & $1280$ & $1248$\\ \hline
		\end{tabular}
	\end{table}
	Also, for self-reciprocal polynomial $f_1\in \fq[x]$, we have the following improvements in manYPoints table in \cite{MP2009}.
	\begin{table}[h!]
		\captionof*{table}{\bf{New entry}}
		\centering
		\begin{tabular}{|c|c|c|c|c|c|c|c|c|c|c|c|c|c|}\hline 
			$q$ & $m_1$ & $m_2$ & $s_1$ & $s_2$ & $f_1$ & $f_2$ & $g$ & $\# \cX(\fqs)$ & $\text{OLB}$ \\ \hline 
			$11$ & $3$ & $3$ & $2$ & $2$ & $x^2+1$ & $x^2+7$ & $16$ & $402$ & $370$\\ \hline
			$11$ & $3$ & $6$ & $0$ & $1$ & $x^2+1$ & $x^2+10$ & $22$ & $462$ & $459$ \\ \hline
			$17$ & $3$ & $6$ & $2$ & $2$ & $x^2+1$ & $x^2+3$ & $37$ & $1224$ & $1179$\\ \hline
		\end{tabular}
	\end{table}
	\begin{table}[h!]
		\captionof*{table}{\bf{New record}}
		\centering
		\begin{tabular}{|c|c|c|c|c|c|c|c|c|c|c|c|c|c|c|}\hline 
			$q$ & $m_1$ & $m_2$ & $s_1$ & $s_2$ & $f_1$ & $f_2$ & $g$ & $\# \cX(\fqs)$ & $\text{OLB}$ \\ \hline 
			$5$ & $3$ & $6$ & $2$ & $5$ & $x^2+1$ & $x^2+4$ & $22$ & $174$ & $168$ \\ \hline
		\end{tabular}
	\end{table}
\end{example}

Analogously to Theorem \ref{theoX1X1}, we have the following result corresponding to another type of fibre product. 

\begin{theorem}\label{theoX2X2} 
Let $i\in\{1, 2\}$. Let $m_i\geq 2$ be a divisor of $q-1$, $s_i$ be an integer $0 \leq s_i < m_i$ and $f_i$ be a separable polynomial in $\fq [x]$ of degree $d_i$ satisfying $f_i(0)\neq 0$ and $(f_i, f_i^*)=(f_1f_1^*, f_2f_2^*)=1$. Then the curve $\cX$ defined by the affine equations
\begin{equation}\label{curveX2X2}
\cX: \left\{
\begin{array}{l}
y_2^{m_2}=\frac{x^{s_2}f_2(x)}{f_2^*(x)}\\
y_1^{m_1}=\frac{x^{s_1}f_1(x)}{f_1^*(x)}
\end{array} \right. 
\end{equation}
has genus 
$$
g=m_1m_2(d_1+d_2)-d_1m_2-d_2m_1+1-\kappa
$$ 
where $\kappa=(m_1m_2, s_1m_2, s_2m_1)$. Further if $\fqs$ is the full constant field of $\fqs(\cX)$, $[\fqs(\cX), \fqs(x)]=m_1m_2$ and $(f_i, x^{q+1}-1)=1$ for $i\in\{1, 2\}$, then the number of $\fqs$-rational points $\#\cX (\fqs)$ of the curve $\cX$ satisfies 
$$
\#\cX (\fqs)\geq m_1m_2(q+1).
$$
\end{theorem}
\begin{example}\label{Example11}
For polynomials $f_1, f_2\in \fq[x]$ satisfying the conditions of Theorem \ref{theoX2X2} and the curve $\cX$ as defined in (\ref{curveX2X2}), we have the following table.
\begin{table}[h!]
\captionof*{table}{\bf{New entry}}
\centering
\begin{tabular}{|c|c|c|c|c|c|c|c|c|c|c|c|c|c|}\hline 
$q$ & $m_1$ & $m_2$ & $s_1$ & $s_2$ & $f_1$ & $f_2$ & $g$ & $\# \cX(\fqs)$ & $\text{OLB}$ \\ \hline 
$13$ & $3$ & $3$ & $0$ & $2$ & $x^2+3x+3$ & $x+4$ & $16$ & $558$ & $464$\\ \hline
$13$ & $4$ & $4$ & $1$ & $1$ & $x+2$ & $x+6$ & $21$ & $568$ & $556$\\ \hline
$ 17 $ & $ 4 $ & $ 4 $ & $ 1 $ & $ 1 $ & $x+3$ & $x+8$ & $21$ & $ 808 $ & $ 794 $\\ \hline
\end{tabular}
\end{table} 

\begin{table}[h!]
\captionof*{table}{\bf{New record}}
\centering
\begin{tabular}{|c|c|c|c|c|c|c|c|c|c|c|c|c|c|}\hline 
$q$ & $m_1$ & $m_2$ & $s_1$ & $s_2$ & $f_1$ & $f_2$ & $g$ & $\# \cX(\fqs)$ & $\text{OLB}$ \\ \hline 
$13$ & $2$ & $4$ & $0$ & $2$ & $x^2+4$ & $x+5$ & $11$ & $444$ & $400$\\ \hline
$13$ & $2$ & $6$ & $0$ & $2$ & $x+3$ & $x+6$ & $13$ & $444$ & $438$\\ \hline
\end{tabular}
\end{table}
\end{example}

\end{document}